\documentclass{amsart}
\usepackage[utf8]{inputenc}
\usepackage{amssymb,bm,amsfonts, stmaryrd, amsmath, amsthm, enumerate, mathtools,comment}
\usepackage[all]{xy}

\makeatletter
\@namedef{subjclassname@2020}{\textup{2020} Mathematics Subject Classification}
\makeatother
\usepackage[colorlinks]{hyperref}
\hypersetup{
 colorlinks=true,
 linkcolor=blue,
 filecolor=magenta, 
 urlcolor=cyan,
}
\usepackage{wrapfig,tikz, tikz-cd}
\usetikzlibrary{arrows}
\usepackage[capitalise]{cleveref}
\crefformat{equation}{(#2#1#3)}
\crefrangeformat{equation}{(#3#1#4--#5#2#6)}
\crefformat{enumi}{(#2#1#3)}
\crefrangeformat{enumi}{(#3#1#4--#5#2#6)}
\usepackage{graphicx}
\usepackage{comment}
\usepackage{amscd,mathabx}
\usepackage{amssymb,setspace}
\usepackage{latexsym,amsfonts,amssymb,amsthm,amsmath,amscd,stmaryrd,mathrsfs, tikz-cd}

\usepackage{mathtools}
\xyoption{all}
\xyoption{arc}

\usepackage{fontawesome}
\usepackage[capitalise]{cleveref}
\crefformat{equation}{(#2#1#3)}
\crefrangeformat{equation}{(#3#1#4--#5#2#6)}
\crefformat{enumi}{(#2#1#3)}
\crefrangeformat{enumi}{(#3#1#4--#5#2#6)}
\newtheorem{theorem}{Theorem}[section]

\newcounter{intro}

\newtheorem{introthm}[intro]{Theorem}
\newtheorem{introcor}[intro]{Corollary}

\theoremstyle{definition}

\newtheorem{chunk}[theorem]{}

\newtheorem{question}[theorem]{Question}

\newtheorem{thm}{Theorem}[section]
\newtheorem{lem}[thm]{Lemma}
\newtheorem{prop}[thm]{Proposition}
\newtheorem{cor}[thm]{Corollary}
\theoremstyle{definition}
\newtheorem{definition}[thm]{Definition}
\newtheorem{ex}[thm]{Example}
\newtheorem{rem}[thm]{Remark}

\newtheorem{notation}[thm]{Notation}

\xyoption{all}
\xyoption{arc}

\def\Ass{\operatorname{Ass}}
\def\ann{\operatorname{ann}}

\def\image{\operatorname{im}}
\def\im{\image}
\def\ker{\operatorname{ker}}

\newcommand{\Hom}{\mathrm{Hom}}
\newcommand{\p}{\mathfrak{p}}

\newcommand{\Soc}{\mathrm{Soc}}

\def\Ext{\operatorname{Ext}}

\newcommand{\Z}{\mathbb{Z}}
\newcommand{\N}{\mathbb{N}}

\numberwithin{equation}{section}

\newcommand{\xra}[1]{\xrightarrow{#1}}

\newcommand{\del}{\partial}

\def\and{{ \text{ and } }}

\newcommand{\Proj}{\mathrm{Proj}}

\newcommand{\Spec}{\mathrm{Spec}}
\newcommand{\coker}{\mathrm{coker}}

\newcommand{\SHom}{\underline{\mathrm{Hom}}}

\newcommand{\uExt}{\underline{\mathrm{Ext}}}
\newcommand{\uTor}{\underline{\mathrm{Tor}}}
\newcommand{\SSoc}{\text{*}\mathrm{Soc}}
\newcommand{\checkgr}{\vee_{\mathrm{gr}}}

\newcommand{\m}{\mathfrak{m}}

\newcommand{\q}{\mathfrak{q}}

\newcommand*{\bigchi}{\mbox{\Large$\chi$}}

\makeindex
\title{Graded-Injective Modules and Bass Numbers of Veronese Submodules}

\author[Taylor Murray]{Taylor Murray}
\address{Department of Mathematics,
University of Nebraska-Lincoln, 336 Avery Hall, Lincoln, NE 68588, U.S.A.}
\email{tmurray11@huskers.unl.edu}

\subjclass[2020]{Primary: Bass numbers, graded Bass numbers, local cohomology modules, graded rings, graded injective modules.} 

\begin{document}
\begin{abstract}
 Let $R$ be a standard graded, finitely generated algebra over a field, and let $M$ be a graded module over $R$ with all Bass numbers finite. Set $(\--)^{(n)}$ to be the $n$-th Veronese functor. We compute the Bass numbers of $M^{(n)}$ over the ring $R^{(n)}$ for all prime ideals of $R^{(n)}$ that are not the homogeneous maximal ideal in terms of the Bass numbers of $M$ over $R$. As an application to local cohomology modules, we determine the Bass numbers of $H_{I\cap R^{(n)}}^i(R^{(n)})$ over the ring $R^{(n)}$ in the case where $H_I^i(R)$ has finite Bass numbers over $R$ and $I$ is a graded ideal.
\end{abstract}
\maketitle


\begin{section}{Introduction}

 All rings are assumed to be Noetherian and commutative with identity, and $k$ will always denote a field. For an $R$-module $M$, let $H_I^i(M)$ denote the \textit{i-th local cohomology module of an $R$-module $M$ supported on an ideal $I\subseteq R$}. Objects in the class of local cohomology modules are often not finitely generated even when $M$ is finitely generated. Therefore, many properties of finitely generated modules are not guaranteed to be present in local cohomology modules. In this manuscript we will be interested in the following finiteness questions asked by Huneke in \cite{MR1165320}:

 \begin{enumerate}
\item When does $H_I^i(R)$ have finite Bass numbers over $R$, and
\item when does $H_I^i(R)$ have finitely many associated primes over $R$?
 \end{enumerate}

 


Unfortunately, Huneke's questions do not have affirmative answers for all local cohomology modules. As shown by Hartshorne, local cohomology modules need not have finite Bass numbers  \cite[pg. 149 -- 151]{Hartshorne_1970}, nor are they guaranteed to have finitely many associated primes as demonstrated in \cite[Theorem 1.2]{MR1922391}, \cite[pg. 8--9]{Singh2004ptorsionEI}, and \cite[Theorem 1.1]{MR2058025}. On the other hand, it is shown in {{\cite[Theorem 2.4(d)]{Lyubeznik_1993}, \cite[Theorem 2.1]{Huneke_Sharp_1993}, and \cite[Corollary 2.14]{Lyubeznik_1997}}} for a regular local ring, $R$, of equal characteristic, all local cohomology modules $H_I^i(R)$ have a number of finiteness properties; in particular, all $H_I^i(R)$ have finitely many associated primes and finite Bass numbers in this setting. 

Remarkably, the finiteness attributes discussed in the previous paragraph descend to direct summands in many cases. To be exact, let $S\to R$ be a map of Noetherian rings that splits as $S$-modules. For $J$ an ideal of $S$,  Núñez-Betancourt in \cite{MR2925808} shows that a number of finiteness properties for $H_{JR}^i(R)$ descend to local cohomology modules $H_{J}^i(S)$. In particular, it is proven in \cite[Theorem 1.1]{MR2925808} that if the local cohomology module $H_{JR}^i(R)$ has finitely many associated primes over $R$, then $H_{J}^{i}(S)$ has finitely many associated primes over $S$. Moreover, if

\begin{enumerate}
\item[(a)] $S\to R$ is a homomorphism of Noetherian rings that splits as $S$-modules,
\item[(b)] $R$ is Cohen-Macaulay,
\item[(c)] $N$ is an $S$-module and $M$ is an $R$-module such that there is a splitting $N\to M$ of $S$-modules, and
\item[(d)] $M$ has finite Bass numbers over $R$,
\end{enumerate}

\noindent then $N$ has finite Bass numbers over $S$ by \cite[Proposition 3.1]{MR2925808}. In particular, for an ideal $I\subseteq S$, the Bass numbers of $H_I^i(S)$ over $S$ are finite, if $H_{IR}^i(R)$ has finite Bass numbers over $R$. 

A special case of \cite[Proposition 3.1]{MR2925808} is when $R$ is an $\N$-graded ring, $S=R^{(n)}$ for some positive integer $n$, $M$ is a graded $R$-module, and $N=M^{(n)}$, where  $(\--)^{(n)}$ denotes the $n$-th Veronese functor. In this setting, we ask the following question:


\begin{question}\label{q1}
Let $R$ be a graded, finitely generated algebra over a field and $n$ a positive integer. If a graded $R$-module $M$ has finite Bass numbers over $R$, then what are the Bass numbers of $M^{(n)}$ over $R^{(n)}$? 
\end{question}

In Section \ref{sec5} we prove the following theorem, yielding an answer to Question \ref{q1}:


\begin{introthm}\label{main theorem 1}
Fix $n\in \N$, and let $R$ be a positively graded $k$-algebra generated by finitely many elements of degree coprime to $n$; for example, this holds if $R$ is standard graded. Assume that $M$ is a graded $R$-module with finite Bass numbers over $R$. Then $M^{(n)}$ has finite Bass numbers over $R^{(n)}$. Moreover, if $\p\in \Spec(R)$, and $\p$ is not the homogeneous maximal ideal, then

\[
\mu(i,\p\cap R^{(n)},  M^{(n)})=\begin{cases}
\mu(i,\p, M) & \p \in \Proj(R), \\
0 & \text{$\p$ is not homogeneous and $i=0$}, \text{ or } \\
\mu(i-1,\p^*, M) & \text{$\p$ is not homogeneous and $i\geq 1$}.
\end{cases}
\]

\noindent where $\p^*=(\{r\in \p~|~ \text{$r$ is homogeneous}\})$ is the largest homogeneous prime ideal contained in $\p$.
\end{introthm}
Tying back to the study of local cohomology, in Section \ref{sec6} we give an application of Theorem \ref{main theorem 1} to local cohomology modules:

\begin{introcor}\label{graded local cohomology thm 2} With $R$ as in Theorem \ref{main theorem 1}, assume $M$ is a graded $R$-module, and $I$ is a graded ideal of $R$. If $H_I^i(M)$ has finite Bass numbers over $R$, then $H_{I\cap R^{(n)}}^i(M^{(n)})$ has finite Bass numbers over $R^{(n)}$. Moreover,

\[
\mu(i,\p\cap R^{(n)}, H_{I\cap R^{(n)}}^j(M^{(n)}))=\begin{cases}
\mu(i,\p, H_{I}^j(M)) & \p\in \Proj(R), \\
0 & \text{$\p$ is not homogeneous and $i=0$, or } \\
\mu(i-1,\p^*, H_{I}^j(M)) & \text{$\p$ is not homogeneous and $i\geq 1$}
\end{cases}.
\]

\

\end{introcor}

The following corollary is a consequence of \cite[Corollary 1.10]{YassemiSiamak1998Wapu}. By appealing to Corollary \ref{graded local cohomology thm 2}, we provide an alternate proof.

\begin{introcor}\label{graded local cohomology cor 2}
 With the same setup of Corollary \ref{graded local cohomology thm 2}, we have

\[
\Ass_{R^{(n)}}\left(H^i_{I\cap R^{(n)}}\left(M^{(n)}\right)\right)=\left\{\p\cap R^{(n)}~|~ \p \in \Ass_R(H_I^i(M))\right\}.
\]
\end{introcor}

 For a graded $R$ module $M$, the works \cite{RFHF} and \cite{GW} show that it is sufficient to study \textit{graded injective resolutions} of $M$ as a graded $R$-module in order to understand the Bass numbers of $M$. This observation leads us study \textit{graded injective hulls}, denoted by $\text{*}E_R(M)$ (see \cite{RFHF}) in Section \ref{sec2}. There are two key results that are employed in the proof of Theorem \ref{main theorem 1}. First, in Section \ref{sec3}, we prove a key result in the setting of Theorem \ref{main theorem 1} that if $M$ does not have the maximal homogeneous ideal of $R$ as an associated prime, then $\text{*}E_R(M)^{(n)}\cong \text{*}E_{R^{(n)}}(M^{(n)})$ (see Corollary \ref{EUC1}); the assumption that the maximal homogeneous ideal of $R$ is not an associated prime of $M$ is necessary (see Example \ref{Jack's Example}). Second, in Section \ref{sec4}, we use \textit{graded Matlis Duality} (see \cite{BH}) to show that if $M$ is an artinian $R$-module, then $M^{(n)}$ has finite Bass numbers over the ring $R^{(n)}$ (see Corollary \ref{fmt}).


\end{section}

\section*{Acknowledgments}
 The author would like to thank his advisor, Jack Jeffries, for helpful and enlightening conversations on the content of this paper. The author would also like to thank Vaibhav Pandey, Josh Pollitz, and Nawaj KC for many helpful comments on earlier drafts of this manuscript. Additionally, the author would like to thank Luis Núñez-Betancourt for bringing \cite{YassemiSiamak1998Wapu} to our attention.

 
\begin{section}{Preliminaries}\label{sec2}
\begin{subsection}{Setting}
Throughout this section we let $R$ be a Noetherian, $\Z$-graded, commutative ring. For a graded $R$-module $M$, we let $M_i$ denote its $i$-th graded piece. A \textit{map of graded $R$-modules} $f:M\to N$ is a $R$-linear map such that $f(M_i)\subseteq N_i$ for all $i\in \Z$; we also call these \textit{graded maps}. We denote the set of \textit{graded prime ideals} of $R$ by $^*\mathrm{Spec}(R)$. For a graded $R$-module $M$ and graded prime $\p$, we denote by $M_{(\p)}$ the localization $S^{-1}M$, with $S=\{r\in R\smallsetminus \p~|~ \text{$r$ is homogeneous}\}$. 

\begin{definition}
Let $R$ be a $\Z$-graded ring. The category of $\Z$-graded $R$-modules, denoted *$R$-mod, is defined to be the category whose objects are graded $R$-modules and morphisms are graded maps of graded $R$-modules.
\end{definition}

For a graded $R$-module $M$ and integer $n$, we define $M(-n)$ to be $M$, as an $R$-module, with grading $[M(-n)]_i=[M]_{i-n}$. Suppose $M$ and $N$ are graded $R$-modules; the ungraded $\Hom_R(M,N)$ may not be a graded $R$-module in a natural way. We define

\[
\SHom_R(M,N):=\bigoplus_{n\in \Z}\Hom_{\text{*}R-\text{Mod}}(M,N(n)),
\]

\noindent which is a graded $R$-module with grading $[\SHom_R(M,N)]_n=\Hom_{\text{*}R-\text{Mod}}(M,N(n))$. For a graded $R$-module $M$, we let $\uExt^\bullet_R(M,\--)$ and $\uExt_R^\bullet(\--,M)$ denote the derived functors of $\SHom_R(M,\--)$ and $\SHom_R(\--,M)$, respectively.

\end{subsection}

\begin{subsection}{Graded Baer's Criterion}

We study the injective objects of *$R$-mod. Recall that in an arbitrary category $\mathcal{C}$ an object $E$ of $\mathcal{C}$ is an \text{injective object} if for every monomorphism $A\to B$ and morphism $A\to E$, there is a morphism $B\to E$ such that the following diagram commutes

\[
\begin{tikzcd}
A \arrow[r]\arrow[d]& B\arrow[dl] \\
E&
\end{tikzcd}.
\]

\

\begin{definition}\label{definition of *inj}
An injective object $E$ in *$R$-mod is called a *injective $R$ module.
\end{definition}

Using only Definition \ref{definition of *inj}, it is easy to prove the following proposition. 

\begin{prop}\label{*shift}
If $M$ is a *injective $R$-module and $k\in \Z$, then $M(-k)$ is a *injective $R$-module.
\end{prop}

Since every monomorphism in *$R$-mod is injective, we attain the following characterization of *injective modules:

\begin{prop}
A graded $R$-module $E$ is *injective if and only if for every graded, injective map $A\to B$ and graded map $A\to E$, there is a graded map $B\to E$ such that the following diagram commutes

\[
\begin{tikzcd}
A \arrow[r]\arrow[d]& B\arrow[dl] \\
C&
\end{tikzcd}.
\]
\end{prop}

In the category of $R$-modules, a useful tool for detecting injective modules is Baer's Criterion. There is a graded analogue of Baer's Criterion, as well.

\begin{thm}[Graded Baer's Criterion]\label{gbc}
A graded $R$-module $E$ is *injective if and only if for every graded ideal $I$ of $R$, $k\in \Z$, and graded map $f:I(-k)\to E$, there is a graded map $g:R(-k)\to E$ such that the following diagram commutes

\[
\begin{tikzcd}
I(-k) \arrow[r, "\subseteq"]\arrow[d, "f"]& R(-k)\arrow[dl, "g"] \\
E&
\end{tikzcd}.
\]
\end{thm}

\begin{proof}

The proof of this is analogous to the one found in \cite[Theorem 1]{Baer_1940}.\qedhere

\end{proof}


Theorem \ref{gbc} can be used to provide examples of *injective $R$-modules:

\begin{cor}
Let $R$ be a $\Z$-graded ring where all homogeneous elements are not zero divisors. Consider the multiplicative set $W=\{r~|~ r\in R \text{ is homogeneous}\}$; then $W^{-1}R$ is a *injective module over $R$.
\end{cor}

\begin{proof}
Let $I(-k)$ be a shift of a graded ideal of $R$. Set $I=(x_{\alpha})_{\alpha\in A}$ with each $x_{\alpha}$ homogeneous. Suppose that $f:I(-k)\to W^{-1}R$ is a graded map. Fix some $x_\alpha$. Define $g:R(-k)\to W^{-1}R$ by $g(1)=x_\alpha^{-1} f(x_{\alpha})$ and extend $R$-linearly. Since $f$ is graded and $x_\alpha$ is homogeneous, it is a quick check to see that $g$ is graded.

\

We need only show that $g|_{I(-k)}=f$. It suffices to see that $f(x_{\beta})=g(x_{\beta})$ for all elements $\beta \in A$. To this end, $f(x_\alpha x_\beta)=f(x_\beta x_\alpha)$, implying $f(x_\alpha)x_\beta=f(x_\beta)x_\alpha$. Thus, we have $g(1)=f(x_\alpha)x_\alpha^{-1}=f(x_\beta)x_\beta^{-1}$. Therefore, for all $\beta\in A$, we have 

\[
g(x_{\beta})=x_\beta g(1)=f(x_\beta)x_\beta^{-1}=f(x_{\beta}).
\]

\

\noindent Hence, $g|_{I(-k)}=f$,  and the graded Baer's Criterion implies $W^{-1}R$ is *injective, as desired.
\end{proof}
\

In general, *injective modules are not injective as ungraded modules. One indication of this is in the graded Baer's criterion;
the maps and ideals considered must be graded, so there are fewer maps that need to extend since there are generally fewer graded maps and ideals than in the non-graded case. Below is an explicit example of a *injective $R$-module that is not an injective $R$-module.

\begin{ex}
Let $R=k[x,y]$ a polynomial ring with the standard grading. The graded module $W^{-1}R$, where $W=\{r\in R~|~ \text{$r$ is homogeneous}\}$, is *injective over the graded ring $R$ by Corollary 2.4; however, it is not injective over $R$.  To see the latter statement, consider the ideal $I=(x-1)$ in $R$. Define the map $f:I\to W^{-1}R$ by $f(x-1)=1$ and extend $R$-linearly. We show that this does not extend to a map $g:R\to W^{-1}R$. If, for sake of contradiction, it did extend, then

\[
1=g(x-1)=(x-1)g(1),
\]

\

\noindent which cannot happen since $x-1$ is not a unit in $R$. Thus, no such extension of $f$ exists.
\end{ex}

\end{subsection}


\begin{subsection}{*Essential Extensions}
Injective hulls play an essential role in the category of $R$-modules, where $R$ is Noetherian. In this case, Matlis in \cite[Theorem 2.5]{Matlis_1958} showed that every injective module is a direct sum of injective hulls. To state the analogous theorem in the graded case, we need some preliminary definitions.

\begin{definition}\cite[\S 1, pg. 189]{RFHF}
\begin{enumerate}
\item [(a)] Let $M\subseteq N$ be graded $R$-modules. We say that $N$ is an *essential extension of $M$ if for every nonzero, graded submodule $U$ of $N$, we have that $U\cap M$ is nonzero.

\item [(b)] Suppose $M\subseteq N$ is an *essential extension such that for every *essential extension $T$ of $N$, we have $T=N$. Then we say that $N$ is a maximal *essential extension of $M$.
\end{enumerate}
\end{definition}

\begin{rem}
To show that $M\subseteq N$ is *essential, it suffices to show that for all nonzero homogeneous $n\in N$, we have that $Rn\cap M\neq (0)$. In what is to follow, we often use this fact often without mention.
\end{rem}

We collect some well-known facts about *essential extensions. The proofs of these are along the same lines as the ungraded case with the proper adjustments.

\begin{lem}\label{*subsets}\label{*ass}

Let $M, N$, and $E$ be graded $R$-modules. 
\begin{enumerate}
\item Suppose that $M\subseteq N \subseteq E$. Then $M\subseteq E$ is *essential if and only if $M\subseteq N$ and $N\subseteq E$ are *essential. 

\item Let $M\subseteq N$ be a *essential extension of graded $R$-modules. Then there is an equality of sets *$\Ass(M)=$*$\Ass(N)$.
\end{enumerate}
\end{lem}

In the ungraded case, maximal essential extensions are unique up to isomorphism and injective. An analogue statement of this fact holds in *$R$-mod.

\begin{thm}[{\cite[Theorem 3.6.2]{BH}}]\label{*iso}
Let $M$ be a graded $R$-module. Then any two maximal *essential extensions are isomorphic. Moreover, the maximal *essential extension is *injective. 
\end{thm}

In light of Theorem \ref{*iso}:

\begin{definition}
 \cite[\S 1, pg. 289--290]{RFHF} We denote by *$E_R(M)$ the maximal *essential  extension of  a graded $R$-module $M$ and say that *$E_R(M)$ is the *injective hull of $M$.
\end{definition}

Just as ungraded injective hulls behave well under localization, *injective hulls behave well under homogeneous localization, as the next few results show.

\begin{lem}\label{*injective localizes}\label{*essential localizes}
Let $R$ be a Noetherian $\Z$-graded ring and $S\subseteq R$ a multiplicative set of homogeneous elements.

\begin{enumerate}

\item If $E$ is a *injective $R$-module, then $S^{-1}E$ is a *injective $S^{-1}R$-module.

\item Suppose $M\to N$ a *essential extension of $R$-modules. Then the homomorphism $S^{-1}M\to S^{-1}N$ is a *essential extension of $S^{-1}R$-modules. In particular, *$E_{S^{-1}R}(S^{-1}M)=S^{-1}\text{*}E_R(M)$.

\end{enumerate}
\end{lem}

\begin{proof}
The proofs of these statements follow as in the ungraded case.
\end{proof}









\begin{prop}\label{*minimal localizes}
Let $R$ be a graded ring, $M$ a graded $R$-module with a minimal *injective resolution $E^\bullet$ over $R$, and $S$ a multiplicative set of $R$ consisting of homogeneous elements. Then $S^{-1}E^\bullet$ is a minimal *injective resolution of $S^{-1}M$ over $S^{-1}R$.
\end{prop}

\begin{proof}
This is a consequence of Lemma \ref{*injective localizes}.
\end{proof}

\begin{prop}[{\cite[Lemma 4.5, Lemma 4.7, Corollary 4.9]{RFHF}}]\label{*hull localization 1}

Suppose that $R$ is a graded ring and that $\p,\q\in \text{*}\Spec(R)$. 

\begin{enumerate}

\item There are isomorphisms of $R_{(\p)}$-modules
\[
\text{*}E_R(R/\p)\cong \text{*}E_{R}(R/\p)_{(\p)}\cong \text{*}E_{R_{(\p)}}(R_{(\p)}/\p R_{(\p)}). 
\]
\item There is an isomorphism of *$R_{(\p)}$-modules
\[
\SHom_R(R/\p, \text{*}E_R(R/\p))_{(\p)}\cong R_{(\p)}/\p R_{(\p)}.
\]
\item There are isomorphisms of of $R_{(\p)}$-modules
\[
\begin{aligned}
\SHom_R(R/\p,\text{*}E_R(R/\q))_{(\p)}&\cong \SHom_{R_{(\p)}}((R/\p)_{(\p)},\text{*}E_R(R/\q)_{(\p)})\\ &\cong \begin{cases} R_{(\p)}/\p R_{(\p)}, & \p=\q\\
0 & \p\neq \q.
\end{cases}
\end{aligned}
\]
\end{enumerate}
\end{prop}






The following proposition is elementary but important.

\begin{prop}\label{soih}
Let $M$ be a $\Z$-graded module and $p\in \text{*Spec}(R)$. Then
\[
\text{*}E_R(M(-k))=\text{*}E_R(M)(-k).
\]
\end{prop}

\begin{proof}
By Proposition \ref{*shift}, we have that $\text{*}E_R(M)(-k)$ is *injective. Moreover, since the inclusion $M \subseteq \text{*}E_R(M)$ is *essential, then $M(-k) \subseteq \text{*}E_R(M)(-k)$ is *essential. 
\end{proof}

Analogous with injective $R$-modules, every *injective graded $R$-module $M$ is isomorphic to shifted copies of indecomposable *injective hulls.

\begin{thm}[{\cite[Theorem 3.6.3]{BH}\cite[Theorem 4.8]{RFHF}}]\label{*structure} 
Let $R$ be a Noetherian, $\Z$-graded ring, and $M$ a *injective $R$-module. Then there is a graded isomorphism of $R$-modules

\[
M\cong \bigoplus_{\substack{\mathfrak{p}\in \mathrm{\text{*}Spec}(R) \\ k\in \Z}}\text{*}E_R(R/\mathfrak{p}(-k))^{\text{*}\eta(\p,k,M)},
\]

\noindent where each $\eta(\p,k,M)$ is a cardinal number.
\end{thm}

\begin{definition}
Let $R$ be a Noetherian, $\Z$-graded ring, and $M$ a *injective graded $R$-module. Suppose $\text{*}E_M^\bullet$ is a minimal *injective resolution of $M$. Then for each $i$, by Theorem \ref{*structure}, we may write

\[
\text{*}E_M^i\cong \bigoplus_{\substack{\mathfrak{p}\in \mathrm{\text{*}Spec}(R) \\ k\in \Z}}\text{*}E_R(R/\mathfrak{p}(-k))^{\text{*}\mu(i,\p,k,M)}.
\]

\

We set $\text{*}\mu(i,\p,M)=\sum_k \text{*}\mu(i,k,\p,M)$ and call it the $i$-th *Bass number of $M$ at $\p$.
\end{definition}

\begin{rem}\label{r1}
There are instances when $\text{*}E_R(R/\p)(-n)\cong \text{*}E_R(R/\p)(-k)$ with $n \neq k$. We shall study when this happens in the next subsection.
\end{rem}




Remarkably, there is no difference between *Bass numbers and Bass numbers for graded modules over graded primes; the following propositions will be used in the proof of Theorem \ref{main theorem 1}.

\begin{prop}[{\cite[Corollary 4.9]{RFHF}}]\label{graded bass numbers}
Let $R$ be a graded ring, $M$ a graded $R$-module, and $\p$ a graded prime ideal of $R$. Then

\begin{enumerate}
\item The group $\Ext_R^i(R/\p, M)_{(\p)}$ is a free, graded $R_{(\p)}/\p R_{(\p)}$-module of rank $\text{*}\mu (i,\p,M)$.
\item $\text{*}\mu (i,\p,M)=\mu(i,\p,M)$.
\end{enumerate}
\end{prop}

In conjunction with Proposition \ref{graded bass numbers} the following result characterizes all Bass numbers of a graded module over any prime in terms of its *Bass numbers.

\begin{prop}[{\cite[Theorem 1.1.2]{GW}}]\label{grade=}
Let $R$ be a $\Z$-graded ring and $M$ a graded $R$-module. Then for all $\p\in \Spec(R)$, we have

\[
\mu(i,\p,M)=
\begin{cases}
\text{*}\mu(i,\p,M) & \p\in \text{*}\Spec(R)\\
\text{*}\mu(i-1, \p, M) & \text{$\p$ is not homogeneous and $i>0$}\\
0 &  \text{$\p$ is not homogeneous and $i=0$},
\end{cases}
\]

\

\noindent where $\p^*=(\{r\in \p~|~\text{ $r$ is homogeneous}\})$.
\end{prop}

\end{subsection}


\begin{subsection}{Isomorphism Classes of Indecomposable *Injective Modules}
In this section we classify when $\text{*}E_R(R/\p)(-n)\cong \text{*}E_R(R/\p)(-k)$ with $n \neq k$. Using this, we will state a ``uniqueness-like" statement for *Bass numbers. We point out that Remark \ref{r1} prohibits a uniqueness statement akin to the uniqueness statement for ordinary Bass numbers.

 \begin{definition} 
 Let $R$ be a graded ring, $M$ a graded module, and $k\in \Z$. Consider the element $m\in M$ and $m\in M(-k)$. We denote by $\deg_{-k}(m)$ the degree of the element $m\in M(-k)$. In other words $\deg_{-k}(m)=\deg(m)+k$.
 \end{definition}

 In the next lemma, we classify when two different shifts of a graded ring are isomorphic as graded $R$-modules. This will aid us in determining when two different shifts of *injective hulls are isomorphic as graded $R$-modules.

\begin{lem}\label{2.3L1}
Let $R$ be a graded ring and $n,k\in \Z$. Then $R(-n)\cong R(-k)$ in *$R$-mod if and only if there is a unit of degree $k-n$ in $R$.  
\end{lem}

\begin{proof}
Suppose that there is a graded isomorphism $\phi: R(-n)\to R(-k)$. Then there is a $r\in R(-n)$ such that $r\phi(1)=\phi(r)=1$. Since $\phi$ is graded, then it follows that $\deg_{-n}(r)=\deg_{-k}(1)=k$. Thus, $\deg(r)+n=k$, implying $\deg(r)=k-n$. Moreover, since $r\phi(1)=1$, we have that $r$ is a unit in $R$.

On the other hand, suppose there is a unit of degree $k-n$ in $R$, so its inverse is of degree $n-k$ in $R$, call it $s$. Define the map $\phi:R(-n)\to R(-k)$ by $\phi(1)=s$ and extending $R$-linearly. Then $\deg_{-n}(1)=n$ and $\deg_{-k}(s)=(n-k)+k=n$. Therefore $\phi$ is graded. Moreover, $\phi$ is readily seen to be an isomorphism since $s$ is a unit in $R$.
\end{proof}

\begin{prop}\label{2.3P1}
Let $R$ be a graded ring, $\p\in \text{*}\Spec(R)$ and $n,k\in \Z$. The following are equivalent:

\begin{enumerate}
\item $\text{*}E_R(R/\p)(-n)\cong \text{*}E_R(R/\p)(-k)$.
\item $R_{(\p)}/\p R_{(\p)}(-n)\cong R_{(\p)}/\p R_{(\p)}(-n)$ as $R_{(\p)}$-modules and hence as $R_{(\p)}/\p R_{(\p)}$-modules.
\item There is a unit of $R_{(\p)}/\p R_{(\p)}$ of degree $k-n$.
\end{enumerate}
\end{prop}

\begin{proof}
$(1) \implies (2)$: Suppose that $\text{*}E_R(R/\p)(-n)\cong \text{*}E_R(R/\p)(-n)$. Since homogeneous localization at $\p$ is exact, Proposition \ref{*hull localization 1} implies $\text{*}E_{R_{(\p)}}(R_{(\p)}/\p R_{(\p)})(-n)$ is isomorphic to $ \text{*}E_{R_{(\p)}}(R_{(\p)}/\p R_{(\p)})(-k)$ as $R_{(\p)}$-modules. Applying $\ann(\p R_{(\p)}, \--)$ to both sides, Proposition \ref{*hull localization 1} implies that 

\[
R_{(\p)}/\p R_{(\p)}(-n)\cong R_{(\p)}/\p R_{(\p)}(-k)
\] 

\

\noindent as $R_{(\p)}$-modules and hence as $R_{(\p)}/\p R_{(\p)}$-modules.

\

$(2)\implies (1)$: Suppose that $R_{(\p)}/\p R_{(\p)}(-n)\cong R_{(\p)}/\p R_{(\p)}(-n)$ as $R_{(\p)}$-modules. Applying $\text{*}E_{R_{(\p)}}(\--)$ to both sides yields that $\text{*}E_{R_{(\p)}}(R_{(\p)}/\p R_{(\p)})(-n)$ is isomorphic to $ \text{*}E_{R_{(\p)}}(R_{(\p)}/\p R_{(\p)})(-k)$ as graded $R$-modules. Hence, Proposition \ref{*hull localization 1} implies the desired isomorphism $\text{*}E_R(R/\p)(-n)\cong \text{*}E_R(R/\p)(-n)$.

\

$(2)\iff (3)$: This is a consequence of Lemma \ref{2.3L1}.\qedhere

\end{proof}

In light of Proposition \ref{2.3P1}, to understand when  $\text{*}E_R(R/\p)(-n)\cong \text{*}E_R(R/\p)(-k)$, we need only understand the structure of $R_{(\p)}/\p R_{(\p)}$. In this setting $R_{(\p)}$ is a *local ring with *maximal ideal $\p R_{(\p)}$.

\begin{definition}\cite[pg. 181]{GW}
A graded ring $R$ is said to be \textit{$H$-simple} if every nonzero homogeneous element of $R$ is invertible.
\end{definition}

For any *local ring  $(R,\m)$, the quotient $R/\m$ is $H$-simple. Indeed, we observe that every nonzero homogeneous element in $R/\m$ is invertible. The following lemma characterizes $H$-simple rings.

\begin{prop}[{\cite[Lemma 1.1.1]{GW}, \cite[Lemma 1.5.7]{BH}}]\label{*local structure}
Let $R$ be a graded ring. The following are equivalent:

\begin{enumerate}
\item $R$ is $H$-simple.
\item $R_0=k$ is a field, and either $R=k$ or $R=k[t,t^{-1}]$, where $t$ is a invertible homogeneous element of $R$, which is transcendental over $k$.
\item Every graded $R$-module is free.
\end{enumerate}
\end{prop}

Proposition \ref{*local structure}, leads to a classification (up to graded isomorphism) of graded $R$-modules.

\begin{lem}\label{2.3L2}
Let $R$ be a graded ring that is H-simple but not a field. As in~Proposition \ref{*local structure}, we set $R=k[t,t^{-1}]$. If $M$ is a graded $R$-module, then 

\[M\cong \bigoplus_{z=0}^{\deg(t)-1} R(-z)^{\oplus \gamma_z},\]

 \noindent where each $\gamma_z$ is some cardinal number.
\end{lem}

\begin{proof}
By Proposition \ref{*local structure}, we have that $M$ is a free $R$-module. Therefore, 

\[
M\cong \bigoplus_{z\in \Z} R(-z)^{\oplus\beta_z}.
\]

\

By Lemma \ref{2.3L1}, we have that  $R(-n)\cong R(-z)$ if and only if there is a unit of degree $z-n$ in $R$. Since $R$ is $H$-simple, all homogeneous units of $R$ have degree divisible by $\deg(t)$. From the definition of degree, we have that $\deg(t^m)=m\deg(t)$ for all $m\in \Z$.  Thus, $R(-n)\cong R(-z)$ if and only if $\deg(t)\mid z-n$. Hence,

\[
M\cong \bigoplus_{z=0}^{\deg(t)-1} R(-z)^{\oplus \gamma_z}
\]

\

\noindent with $\gamma_z=\sum_{n\in \Z} \beta_{z+\deg(t)n}$.
\end{proof}

For a $H$-local ring $R$ and $R$-module $M$, the cardinal numbers $\gamma_z$ in Lemma \ref{2.3L2} are unique, as is shown by the following lemma.

\begin{lem}\label{2.3L3}
Let $R$ be a graded ring that is H-simple. Write $R=k[t,t^{-1}]$ as in Proposition \ref{*local structure}. If 

\[
\bigoplus_{z=0}^{\deg(t)-1} R(-z)^{\oplus \gamma_k} ~\cong\bigoplus_{z=0}^{\deg(t)-1} R(-z)^{\oplus \beta_k},
\]

\noindent then $\gamma_z=\beta_z$.
\end{lem}

\begin{proof}
Since $R=k[t,t^{-1}]$, then $[R]_i=0$ for $i\in \Z$ such that $\deg(t)$ is not divisible by $i$. Thus, for $n,i\in \Z$ we have $[R(-n)]_i=[R]_{i-n}=0$ when $\deg(t)$ is not divisible by $i-n$. Thus, for each $z$ with $0\leq s \leq \deg(t)-1$, we have 

\[
\left[ \bigoplus_{z=0}^{\deg(t)-1} R(-z)^{\oplus \gamma_k} \right]_s= \bigoplus_{z=0}^{\deg(t)-1} \left[R(-z)^{\oplus \gamma_k}\right]_s =[R(-s)^{\oplus \gamma_s}]_s=k^{\oplus \gamma_s}.
\]

\

Similarly,

\[
\left[ \bigoplus_{z=0}^{\deg(t)-1} R(-z)^{\oplus \beta_k} \right]_s= \bigoplus_{z=0}^{\deg(t)-1} \left[R(-z)^{\oplus \beta_k}\right]_s =[R(-s)^{\oplus \beta_s}]_s=k^{\oplus \beta_s}.
\]

\

Since 

\[
\bigoplus_{z=0}^{\deg(t)-1} R(-z)^{\oplus \gamma_k} ~\cong\bigoplus_{z=0}^{\deg(t)-1} R(-z)^{\oplus \beta_k},
\]

\

\noindent as graded $R$-modules, we have $k^{\oplus \gamma_s}\cong k^{\oplus \beta_s}$ as $[R]_0=k$-modules. In particular, we have $\gamma_s=\beta_s$, as desired.
\end{proof}

We are now in a position to prove a ``uniqueness analogue" to Theorem \ref{*structure}. First, we introduce some notation.

\begin{notation}
Let $R$ be a graded ring and $\p$ a graded prime ideal of $R$ such that $R_{(\p)}/\p R_{(\p)}$ is not a field. By Proposition \ref{*local structure}, there exist a field $k_\p$ and an invertible homogeneous element $t_\p$ of $R_{(\p)}/\p R_{(\p)}$, such that $R_{(\p)}/\p R_{(\p)}=k_\p[t_\p,t_\p^{-1}]$.
\end{notation}

\begin{thm}\label{2.3T1}
Let $R$ be a Noetherian, $\Z$-graded ring and $M$ a *injective $R$-module. Then there is a graded isomorphism of $R$-modules

\[
M\cong \bigoplus_{\substack{\mathfrak{p}\in \mathrm{\text{*}Spec}(R)}}\bigoplus_{k=0}^{\deg(t_\p)-1} \text{*}E_R(R/\mathfrak{p}(-k))^{\oplus\text{*}\eta(k,\p,M)}
\]

\noindent where each $\text{*}\mu(k,\p,M)$ is a cardinal number. Moreover, the $\text{*}\mu(k,\p,M)$ are unique in the following sense: if there is another decomposition 

\[
M\cong \bigoplus_{\substack{\mathfrak{p}\in \mathrm{\text{*}Spec}(R)}}\bigoplus_{k=0}^{\deg(t_\p)-1} \text{*}E_R(R/\mathfrak{p}(-k))^{\oplus\text{*}\gamma(k,\p,M)},
\]

\

\noindent then $\text{*}\mu(k,\p,M)=\text{*}\gamma(k,\p,M)$ for all $k$ and all $\p\in \text{*}\Spec(R)$.
\end{thm}

\begin{proof}
By Proposition \ref{*structure}, we may write 

\[
M\cong \bigoplus_{\substack{\mathfrak{p}\in \mathrm{\text{*}Spec}(R) \\ k\in \Z}}\text{*}E_R(R/\mathfrak{p}(-k))^{\oplus\text{*}\eta(k,\p,M)}.
\]

\

\noindent Let $\p \in \text{*}\Spec(R)$. Then by Proposition \ref{2.3P1} and Lemma \ref{2.3L2}, we have

\[
M\cong \bigoplus_{\p\in \text{*}\Spec(R)} \bigoplus_{k=0}^{\deg(t_\p)-1}\text{*}E_R(R/\mathfrak{\p}(-k))^{\oplus\text{*}\eta(k,\p,M)},
\]

\

\noindent with $\text{*}\mu(k,\p,M)=\sum_n \eta(k+n\deg(t_\p),\p,M)$. Suppose there is another decomposition of $M$

\[
M\cong \bigoplus_{\substack{\mathfrak{p}\in \mathrm{\text{*}Spec}(R)}}\bigoplus_{k=0}^{\deg(t_\p)-1} \text{*}E_R(R/\mathfrak{p}(-k))^{\oplus\text{*}\gamma(k,\p,M)}.
\]

\

\noindent By Lemma \ref{*hull localization 1}, we have that 

\[
\begin{aligned}
\SHom_{R_{(\p)}}((R/\p)_{(\p)},M_{(\p)})&\cong \bigoplus_{k=0}^{\deg(t_\p)-1}\left(R_{(\p)}/\p R_{(\p)}(-k)\right)^{\oplus\text{*}\eta(k,\p,M)}\\
&\cong \bigoplus_{k=0}^{\deg(t_\p)-1} \left(R_{(\p)}/\p R_{(\p)}(-k)\right)^{\oplus\text{*}\gamma(k,\p,M)}.
\end{aligned}
\]

\

\noindent Therefore, by Lemma \ref{2.3L3}, we have $\text{*}\eta(k,\p,M)=\text{*}\gamma(k,\p,M)$ for all $k\in \Z$. \qedhere \end{proof}

We now apply Theorem \ref{2.3T1} to the study of graded Bass numbers.

\begin{cor}\label{graded bass number dimension}
Let $R$ be a Noetherian, $\Z$-graded ring and $M$ a graded $R$-module. Let $\p \in \text{*}\Spec(R)$. By Lemma \ref{2.3L2}, we may write 

\[
\uExt_{R_{(\p)}}^i((R/\p)_{(\p)}, M_{(\p)}))\cong \bigoplus_{z=0}^{\deg(t)-1} (R/\p)_{(\p)}(-z)^{\oplus \gamma_{z,i}},
\] 

\

\noindent where the cardinal numbers $\gamma_{z,i}$ are unique. Then for $0\leq z \leq \deg(t)-1$ and a minimal *injective resolution of $M$, $E^\bullet_M$ , we have 

\[
\text{*}\mu(z,\p,E^i_M)=\gamma_{k,i}.
\]

\

\noindent In particular, the numbers $\text{*}\mu(z,\p, E^i_M)$ depend only on $M$.
\end{cor}

\begin{proof}
Suppose that $(E^\bullet,\del)$ is a minimal *injective resolution of $M$. Let $\p$ be a graded prime ideal of $R$. After homogeneous localization at $\p$, Proposition \ref{*minimal localizes} implies that $(E^\bullet_{(\p)})$ is a minimal *injective resolution for $M_{(\p)}$ over $R_{(\p)}$. By applying Theorem \ref{2.3T1}, for each $i$, we may write

\[
E^i=\bigoplus_{\substack{\q\in \text{*}\Spec(R)}}\bigoplus_{z=0}^{\deg(t_\q)-1}\text{*}E_R(R/\q)(-z)^{\oplus \mu(z,\mathfrak{q},E^i)}
\]

\

\noindent Then by Lemma \ref{*hull localization 1}, we have

\[
\begin{aligned}
\underline{\Hom}_{R_{(\p)}}\left(~(R/\p)_{(\p)},~\left(E^i\right)_{(\p)}\right)&\cong \ann_{R_{(\p)}}\left(\p R_{(\p)}, \left(E^i\right)_{(\p)}\right)\\&=\Soc_{R_{(\p)}/\p R_{(\p)}}((E^i)_{(\p)})\\ &\cong\bigoplus_{z=0}^{\deg(z_\p)-1}R_{(\p)}/\p R_{(\p)}(-z)^{\mu(z,\mathfrak{p},E^i)}. 
\end{aligned}
\]

\noindent  We now show that the differential of the complex $\underline{\Hom}_{R_{(\p)}}\left(~(R/\p)_{(\p)},~\left(E^\bullet\right)_{(\p)}\right)$ is 0. For each $i\geq 1$, we have a commutative diagram

\begin{center}
\[
\begin{tikzcd}
\underline{\Hom}_{R_{(\p)}}(R_{(\p)}/\p R_{(\p)},~E^{i-1}_{(\p)})\arrow[d]\arrow[r, "\overline{\del^i_{(\p)}}"] &\underline{\Hom}_{R_{(\p)}}(R_{(\p)}/\p R_{(\p)},~E^i_{(\p)})\arrow[d] \\
\Soc_{R_{(\p)}/\p R_{(\p)}}((E^{i-1})_{(\p)}) \arrow[r, "\del^i_{(\p)}"] & \Soc_{R_{(\p)}/\p R_{(\p)}}((E^i)_{(\p)}) 
\end{tikzcd}
\]
\end{center}
\

\noindent where the vertical arrows are isomorphisms sending $\phi$ to $\phi(1)$, and the horizontal maps on the second row are restriction maps of the $\del_i\otimes 1_{R_{(\p)}}$.

\

Note that $E^i=E_R\left( \frac{E^{i-1}}{\im(\del^{i-1})}\right)$. For each $i$, we then have a commutative diagram

\begin{center}
\[
\begin{tikzcd}
& \frac{E^{i-1}}{\im(\del^{i-1})}\arrow[dr, "g_i"]& \\
E^{i-1}\arrow[ur,"\pi_i"]\arrow[rr, "\del^i"] && E^i,
\end{tikzcd}
\]
\end{center}

\

\noindent where $\pi_i$ is the canonical surjection map and $g_i$ is the *essential extension. In the case $i=0$, $E^{-1}=M$. Thus, since homogeneous localization is exact

\begin{center}
\[
\begin{tikzcd}
& \frac{E^{i-1}_{(\p)}}{\im(\del^{i-1}_{(\p)})}\arrow[dr, "g_i"]& \\
E^{i-1}_{(\p)}\arrow[ur,"\pi_i"]\arrow[rr, "\del^i_{(\p)}"] && E^i_{(\p)},
\end{tikzcd}
\]
\end{center}

\noindent commutes, where $\pi_i$ is the canonical surjection, and $g_i$ an *essential injection by Lemma \ref{*essential localizes}. Moreover, $\Soc\left( \frac{E^{i-1}_{(\p)}}{\im(\del^{i-1}_{(\p)})}\right)=\Soc(E^i_{(\p)})$. As $\pi_i$ is surjective, we then have that $\Soc(E^i_{(\p)})\subseteq \im(\del^i_{(\p)})$. Since $E^\bullet_{(\p)}$ is exact, this that $\Soc(E^i_{(\p)})$ is contained in $ \ker(\del^{i+1}_{(\p)})$ for all $i$. Therefore, the maps in 

\begin{center}
\[
\begin{tikzcd}
\Soc_{R_{(\p)}/\p R_{(\p)}}((E^0)_{(\p)}) \arrow[r, "\del^1_{(\p)}"] & \Soc_{R_{(\p)}/\p R_{(\p)}}((E^1)_{(\p)})\arrow[r, "\del^2_{(\p)}"]& \Soc_{R_{(\p)}/\p R_{(\p)}}((E^2)_{(\p)})\arrow[r] & \cdots
\end{tikzcd}
\]
\end{center}

\

\noindent are zero. Therefore, the complex $\underline{\Hom}_{R_{(\p)}}\left(~(R/\p)_{(\p)},~\left(E^\bullet\right)_{(\p)}\right)$ has a differential of 0. Thus,

\[
\begin{aligned}
\underline{\Ext}_R^i(R/\p, M)_{*\p}&\cong \underline{\Ext}_{R_{(\p)}}^i((R/\p)_{(\p)}, M_{(\p)})\\ &\cong \underline{\Hom}_{R_{(\p)}}\left(~(R/\p)_{(\p)},~\left(E^i\right)_{(\p)}\right)\\&\cong \bigoplus_{z=0}^{\deg(t_\p)-1} (R/\p)_{(\p)}(-z)^{\oplus \mu(z,\p,E^i)}.
\end{aligned}
\]

\

\noindent Therefore, by Lemma \ref{2.3L3}, we have $\mu(z,\p,E^i)=\gamma_{k,i}$, as desired.\qedhere

\end{proof}

In the spirit of Corollary \ref{graded bass number dimension}, we make the following definition.

\begin{definition}
Let $R$ be a graded Noetherian ring, $M$ a graded $R$-module, and the complex $E^\bullet_M$ a minimal *injective resolution of $M$.  For $\p\in \text{*}\Spec(R)$ and integer $k$ such that $0\leq k \leq \deg(t_\p)-1$, we set 

\[
\text{*}\mu(i,k,\p,M):=\text{*}\mu(k,\p, E_M^i).
\]

\

\noindent We call $\text{*}\mu(i,k,\p,M)$ the $(i,k)$-th graded Bass number with respect to $M$ and note that it is uniquely determined by $M$ by Corollary \ref{graded bass number dimension}.
\end{definition}

\noindent The proof of the following corollary is contained in the proof of Theorem \ref{graded bass number dimension}; we copy it down here for convinence.

\begin{cor}\label{foward dir}
Let $R$ be a Noetherian graded ring and $M$ a graded $R$-module. If the complex $(E^\bullet,\del)$ is a minimal *injective resolution of $M$, then the differential of the complex $\underline{\Hom}_{R_{(\p)}}\left(~(R/\p)_{(\p)},~\left(E^\bullet\right)_{(\p)}\right)$ is 0 for all $\p\in \text{*}\Spec(R)$. 
\end{cor}

\noindent In the setting of Corollary \ref{foward dir} if we further assume that all the *Bass numbers of the $R$-module $M$ are finite, then the converse holds.

\begin{cor}\label{classification of minimal}
Let $R$ be a Noetherian graded ring and $M$ a graded $R$-module. Assume that all the *Bass numbers of $M$ are finite. Then $(E^\bullet,\del)$ is a minimal *injective resolution of $M$ if and only if the complex $\underline{\Hom}_{R_{(\p)}}\left(~(R/\p)_{(\p)},~\left(E^\bullet\right)_{(\p)}\right)$ has a differential of 0 for all $\p\in \text{*}\Spec(R)$. 
\end{cor}

\begin{proof}
It suffices to prove the reverse direction. Let $(E^\bullet, \del)$ be a *injective resolution of $M$ such that the differential of the complex $\underline{\Hom}_{R_{(\p)}}\left(~(R/\p)_{(\p)},~\left(E^\bullet\right)_{(\p)}\right)$ is 0 for all $\p\in \text{*}\Spec(R)$. By Theorem \ref{*structure} and Theorem \ref{2.3L2}, we may write

\[
E^i \cong \bigoplus_{\substack{\mathfrak{p}\in \mathrm{\text{*}Spec}(R)}}\bigoplus_{z=0}^{\deg(t_\p)-1}\text{*}E_R(R/\mathfrak{p})(-z)^{\text{*}\eta(i,\p,z, E^i)}.
\]

\

\noindent By Lemma  \ref{*hull localization 1}, for every fixed $\p \in \text{*}\Spec(R)$, we have

\[
\begin{aligned}
\underline{\Hom}_{R_{(\p)}}\left(~(R/\p)_{(\p)},~\left(E^i\right)_{(\p)}\right)&\cong \ann_{R_{(\p)}}(\p R_{(\p)}, \left(E^i\right)_{(\p)})\\ &=\SSoc_{R_{(\p)}}\left(\left(E^i\right)_{(\p)}\right)\\
&=\bigoplus_{z\in \Z}(R/\p)_{(\p)}(-z)^{\eta(i,\p,z, E^i)}.
\end{aligned}
\]

\noindent As the differential of the complex $\underline{\Hom}_{R_{(\p)}}\left(~(R/\p)_{(\p)},~\left(E^\bullet\right)_{(\p)}\right)$ is 0 for all graded primes $\p$, and $(E^\bullet)_{(\p)}$ is a *injective resolution for $M_{(\p)}$ over $R_{(\p)}$, we have that

\[
\underline{\Ext}_{R_{(\p)}}^i((R/\p)_{(\p)},M_{(\p)})\cong \bigoplus_{z=0}^{\deg(t_\p)-1} k(-z)^{\oplus \text{*}\eta(\p,z, E^i)}.
\]

\

\noindent So, by Theorem \ref{graded bass number dimension},   

\[
\text{*}\mu(i,\p,z,M)=\text{*}\mu(\p,z, E^i)
\]

\

\noindent  for all $\p \in \text{*}\Spec(R)$. Now let $I^\bullet _M$ be a minimal *injective resolution of $M$. Applying the graded analogue of Exercise 11.1.11 in \cite{MR}, yields an injection of graded complexes $f:I^\bullet_M \to E^\bullet_M$. In particular each $I^i_M$ is a direct summand of *$E^i_M$. Since the *Bass numbers of $M$ are finite and 
$\text{*}\mu(i,\p,z,M)=\text{*}\mu(\p,z, E^i)$, we have that the map $f$ is an isomorphism, so $E^\bullet$ is minimal, as desired.\qedhere

\end{proof}

\

\end{subsection}
\end{section}


\begin{section}{Veronese Submodules}\label{sec3}
\begin{subsection}{The Veronese Functor and its Properties}
In this short subsection we let $R$ be a $\Z$-graded ring and $M$ a graded $R$-module. For $n\in \N$, we define $R^{(n)}$ to be the graded ring with grading 

\[
(R^{(n)})_i=\begin{cases}
R_i & \text{$n$ divides $i$} \\
0 & \text{else}
\end{cases}
\]

\

\noindent We note that $R$ is a graded $R^{(n)}$-module via the usual multiplication. More generally, given any graded $R$-module $M$, we define $M^{(n)}$ to be the  *$R^{(n)}$-module with grading

\[
(M^{(n)})_i=\begin{cases}
M_i & \text{$n$ divides $i$}\\
0 & \text{else}
\end{cases},
\]

\

\noindent where the graded $R^{(n)}$-module structure on $M^{(n)}$ is inherited by the graded $R$-module structure on $M$. Then for each $n\in \N$, we have a functor

\[\
(\--)^{(n)}: \text{*}R-\text{mod}\to \text{*}R^{(n)}-\text{mod}
\]

\

\noindent that on objects takes a graded $R$-module $M$ to the *$R^{(n)}$-module $M^{(n)}$, and takes a *$R$-morphism $f:M\to N$ to the *$R^{(n)}$-morphism $f|_{M^{(n)}}:M^{(n)}\to N^{(n)}$. We call the functor $(\--)^{(n)}$ the \textit{$n$-th Veronese functor}. We collect a few basic facts about the Veronese functor.

\begin{prop}\label{*veronese exact}
Let *$R$ be a $\Z$-graded ring and $n\in \Z$. The $n$-th Veronese functor is exact. 
\end{prop}

\begin{proof}
Let $0\to M' \xra{f} M \xra{g} M''\to 0$ be a short exact sequence of $\text{*}R$-modules. Since $(f)^{(n)}$ restricts the domain to $M'^{(n)}$ and $f$ is injective, then $(f)^{(n)}$ is injective, well. On the other hand, Let $m\in M''^{(n)}$. Then $m''\in M''$; the surjectivity of $g$ implies there is an $m\in M$ such that $g(m)=m''$. Since $g$ is graded and $n\mid \deg(m'')$, then $n\mid \deg(m)$. Therefore, $(g)^{(n)}(m)=m''$. 

Lastly, we show exactness at $M^{(n)}$. Let $m\in \ker((g)^{(n)})$; then $m\in \ker(g)$. Hence, there is an $m'\in M'$ such that $f(m')=m$. As $f$ is a graded map, $n\mid deg(m')$, implying $m'\in M'^{(n)}$. Therefore, $(f)^{(n)}(m')=m$, so $\ker((g)^{(n)})\subseteq \im((f)^{(n)})$. Now suppose that $m\in \im((f)^{(n)})$, then there is a $m'\in M'$ such that $(f)^{(n)}(m')=m$. Therefore, $g(f(m'))=0$, implying $m\in \ker(g)$. Since $m\in M^{(n)}$, it is immediate that $m\in \ker((g)^{(n)})$, completing the proof.
\end{proof}

We omit a proof of the following proposition, and note that a proof follows from elementary properties of graded $R$-modules and the $n$-th Veronese functor.

\begin{prop}\label{*vcwds}
Let *$R$ be a $\Z$-graded ring, $n\in \Z$, $\{M_i\}_{i\in I}$ a collection of graded $R$-modules. Then

\[
\left(\bigoplus_{i\in I}M_i\right)^{(n)}=~\bigoplus_{i\in I}M_i^{(n)}.
\].
\end{prop}

\end{subsection}


\begin{subsection}{Connections to *Injective Hulls}
\
In this subsection, we study how the Veronese functor behaves with *injective hulls. First, we prove the following lemma.

\begin{lem}[{\cite[Remark 1.9]{RS}}]\label{*reih}
Fix $k\in \Z$. Let $\mathfrak{p}$ be in $\text{*Spec}(R)$ and $y\in R \smallsetminus \mathfrak{p}$ be any homogeneous element of degree $n$. Then the graded map that is multiplication by $y$ on $\text{*}E_R(R/\mathfrak{p})(-k)$, namely 

\[
\text{*}E_R(R/\mathfrak{p})(-k) \xra{\cdot y} \text{*}E_R(R/\mathfrak{p})(-k+n),
\]

\

\noindent is an isomorphism. In particular, $y$ is homogeneous non zerodivisor on $\text{*}E_R(R/\mathfrak{p})(-k)$.
\end{lem}

\begin{proof}
Set $E:=\text{*}E_R(R/\mathfrak{p})(-k)$. By Proposition \ref{soih}, we have $E=\text{*}E_R(R/\mathfrak{p}(-k))$.  Consider the multiplicative set: $W=\{r\in R~|~ r\notin \mathfrak{p} \text{ and $R$ is homogeneous}\}$ and set $W^{-1}(R/\mathfrak{p}(-k))=N$, which is a *$R$-module. Since $R/\mathfrak{p}(-k) \subseteq N$ is *essential, then a copy of $N$ is contained in $E$. Note that multiplication by $y$ is one-to-one on $N$ since $y\in W$ (in fact $yN=N$ since $y\in W$). Set $K=\ker (E\xra{\cdot y} E(+n))$. Then as $K$ is a graded submodule of $E$, we have that $K\cap N\neq 0$ or $K=0$. However, we must have $K\cap N=0$, so $K=0$; hence, multiplication on $E$ by $y$ is one-to-one. 

As $yE\cong E(-n)$ and $E(-n)$ is *injective, then $yE$ is *injective. Since $yN=N$, we have that $N \subseteq yE \subseteq E$. As $N \subseteq yE \subseteq E$ and $N \subseteq E$ is *essential, then it follows from Lemma \ref{*subsets}  that  $yE \subseteq E$ is *essential. By Theorem \ref{*iso} it follows that $yE=E$. Thus,

\[
E \xra{\cdot y} E(-n),
\]

\noindent is an isomorphism, as desired.
\end{proof}

An interesting consequence of Lemma \ref{*reih} is, under the hypothesis above, if the intersection $R\setminus \mathfrak{p} \cap R_1$ is nonempty, then *$E_R(R/\mathfrak{p})\cong \text{*}E_R(R/\mathfrak{p})(-k)$ for all $k\in \Z$. We copy this down for reference.

\begin{cor}\label{IC2}
Let $R$ be a $\Z$- graded ring. Let $\mathfrak{p}\in \text{*Spec}(R)$ such that $R\setminus \mathfrak{p}\cap R_1$ is nonempty. Then

\[
\text{*}E_R(R/\mathfrak{p})\cong \text{*}E_R(R/\mathfrak{p})(-k)
\]

\

\noindent for all $k\in \Z$.
\end{cor}









For the rest of the paper, we will refine the class of $\Z$-graded rings that we consider. In particular, we will be interested in the case where $R$ is a (not necessarily standard) graded, finitely generated $k$-algebra, where $k$ is any field. As a starting point in our investigation of how a Veronese functor interacts with *$R$-injective modules, we analyze how they behave with shifts of *injective hulls. We begin with a key lemma.

 \begin{lem}\label{VUL}
Fix $n\in \N$, and let $R$ be a positively graded $k$-algebra generated by finitely many elements of degree coprime to $n$. Let $z\in \Z$ and $\p \in \Proj(R)$, where we set $\Proj(R)=\text{*}\Spec(R)\smallsetminus \{\m\}$.  Then $E_R(R/\p(-z))^{(n)}$ is a *injective $R^{(n)}$-module.
 \end{lem}

 \begin{proof}
Assume $\p \in \Proj(R)$. We show that $E_R(R/\p(-z))^{(n)}$ is a *injective $R^{(n)}$-module using graded Baer's criterion. Our strategy amounts to showing that any graded map $\phi: I(-j)\to E_R(R/\p(-z))^{(n)}$ can be extended to a graded $R$-linear map $\overline\phi:IR(-j)\to E_R(R/\p(-z))$. Once we show that $\phi$ extends to $\overline{\phi}$, then there is a map $\psi:R(-j)\to E_R(R/\p(-z)$ that restricts to $\overline{\phi}$ on $IR(-j)$ since $E_R(R/\p(-z))$ is *injective. Let $\theta=\psi|_R^{(n)}$. Then 

\[
\theta|_{I(-j)}=\psi|_{I(-j)}=\overline{\phi}|_{I(-j)}=\phi.
\]

\

\noindent In particular, for any graded homomorphism  $\phi: I(-j)\to E_R(R/\p(-z))^{(n)}$, there is a lift $\theta: R^{(n)}(-j)\to  E_R(R/\p(-z))^{(n)}$ of $\phi$. Therefore, by the graded Baer's Criterion, $E_R(R/\p)^{(n)}$ is *injective over $R^{(n)}$.

Now, we prove that any graded map $\phi: I(-j)\to E_R(R/\p(-z))^{(n)}$ can be extended to a graded map $\overline{\phi}:I(-j)R\to E_R(R/\p(-z))$. let $I=(f_1,\hdots,f_k)$ be a graded ideal of $R^{(n)}$. Consider a graded map $\phi: I(-j)\to E_R(R/\p)(-z)$. Then the graded $R^{(n)}$-module  $I(-j)$ is generated by $\{f_1,\hdots, f_k\}$. Similarly,  $IR(-j)$ is generated by $\{f_1,\hdots,f_k\}$ as a graded $R$-module.

For a homogeneous element $t\in IR(-j)$, since $IR(-j)$ is a graded $R$-module, we may write

\[
t=\sum_ir_if_i,
\]

\noindent where each $r_i\in R$ is homogeneous in $R$ and $\deg_{-j}(r_if_i)=\deg_{-j}(t)$ for all $i$. We define the function $\overline{\phi}$ on homogeneous elements of $IR(-j)$ by 

\[
\overline{\phi}(t)=\sum_ir_i\phi(f_i).
\]

\

\noindent  We then extend, $R$-linearly, $\overline{\phi}$ to all of $IR(-j)$.  Since every element of $IR(-j)$ is uniquely expressed as a sum of homogeneous elements of $IR(-j)$, to show that $\overline{\phi}$ is well defined, it suffices to show that $\overline{\phi}$ is well defined on the set of homogeneous elements of $IR(-j)$. We first show that $\overline{\phi}|_{I(-j)}=\phi$. To this end, let $f\in IR^{(n)}(-j)$ be such that $t$ is homogeneous. Then we may write 

\[
\sum_i p_if_i=t=\sum_ir_if_i,
\]

\noindent where each $r_i,p_i\in R$ are homogeneous and $\deg_{-j}(r_if_i)=\deg_{-j}(t)=\deg_{-j}(p_jf_j)$. Since $f,f_i\in R^{(n)}$, we see that $r_i,p_i\in R^{(n)}$. Indeed,

\[
\deg_R(t)-j=\deg_{-j}(t)=\deg_{-j}(p_if_i)=\deg_R(p_i)+\deg_{-j}(f_i)=\deg_R(p_i)+\deg_R(f_i)-j.
\]

\noindent Hence, $\deg_R(x)=\deg_R(f_i)+\deg_R(p_i)$, implying $n$ divides $\deg_R(p_i)$, so $p_i\in R^{(n)}$ for all $i$. Similarly, we have $r_i\in R^{(n)}$ for all $i$. Therefore, 

\[
\overline{\phi}(t)=\overline{f}(\sum_i p_if_i)=\sum p_i\phi(f_i)=\sum \phi(p_if_i)=\phi(t).
\]

\

\noindent Similarly, 

\[
\overline{\phi}(t)=\overline{\phi}(\sum_j r_if_i)=\sum r_i\phi(f_i)=\sum \phi(r_if_i)=\phi(t).
\]

\noindent Hence, $\overline{\phi}$ is well defined on the homogeneous elements of $I(-j)\subseteq IR(-j)$ and is equal to $\phi$ on this set.

Now, let $t\in I(-j)$ be any element (here we note that $t\in R^{(n)}(-j)$; it just may not be homogeneous). We may uniquely write $f=\sum r_i$ with each $r_i$ homogeneous in $R(-j)$ with $\deg_{-j}(r_i)=i$. Since $t\in I(-j)$, we must have each $r_i\in R^{(n)}(-j)$. As $I$ is graded, each $r_i\in I \subseteq R^{(n)}$. Hence, by definition of $\overline{\phi}$ and the work above, we have

\[
\overline{\phi}(x)=\sum\overline{\phi}(r_i)=\sum \phi(r_i)=\phi(x).
\]

\

\noindent Therefore, $\overline{\phi}|_{I(-j)}=\phi$. 

\

 Suppose $t\in IR(-j)$ is homogeneous of degree $d$. We write

\[
\sum_ip_if_i=t=\sum_ir_if_i,
\]

\

\noindent with $r_i,p_i$ homogeneous elements in $R$ . Since $\p\in \Proj(R)$ all $k$-algebra generators of $R$ have degrees co-prime to $n$, there is a $k$-algebra generator, call it $y$, with $y\notin \p$ and $y\notin R^{(n)}$. We show that there is an $M\in \N$, such that $y^Mt\in R^{(n)}(-j)$. To this end, $y^Mt\in R^{(n)}(-j)$ if and only if $n\mid \deg(y^Mt)$. Since $\deg(y)$ and $n$ are co-prime, there are integers $c_1$ and $c_2$ such that $c_1\deg(y)+c_2n=j-d$. Without loss of generality, we may assume that $c_1\geq 1$. Setting $M=c_1$, yields 

\[
\deg_R(y^Mt)=\deg_R(y^M)+\deg_R(t)=M\deg(y)+d-j=-c_2n.
\]

\



\noindent Therefore, the claim follows. Hence,

\[
y^M\left( \sum_ip_if_i\right)=y^M\left(\sum_ir_if_i\right)
\]

\noindent is an equality in $R^{(n)}(-j)$. Since $\overline{\phi}|_{I(-j)}=\phi$ is well defined

\[
\overline{\phi}\left(y^M\left( \sum_ip_if_i\right)\right)=\overline{\phi}\left(y^M\left(\sum_ir_if_i\right)\right).
\]

\

\noindent Thus,

\[
\phi\left(y^M\left( \sum_ip_if_i\right)\right)=\phi\left(y^M\left(\sum_ir_if_i\right)\right).
\]

\

\noindent As $y^M\sum_ip_if_i$ and $y^M\sum_ir_if_i$ are in $I(-j)$ and each $f_i\in R^{(n)}(-j)$, then each $y^Mp_i$ and $y^Mr_i$ are in $R^{(n)}$. Therefore

\[
y^M\left(\sum_ip_i\phi(f_i)\right)=y^M\left(\sum_ir_i\phi(f_i)\right).
\]

\

\noindent As $y$ is not in $\p$, we have that $y$ is not a zero divisor in $E_R(R/\p)(-k)$. Hence, $y^M$ is not a zero divisor in $E_R(R/p(-k))^{(n)}$; therefore the above equality implies that 

\[
\overline{\phi}(\sum_ip_if_i)=\sum_ip_i\phi(f_i)=\sum_ir_i\phi(f_i)=\overline{\phi}(\sum_ir_if_i).
\]

\noindent Hence, $\overline{\phi}$ is well defined on homogeneous elements of $IR(-j)$, implying $\overline{\phi}$ is well defined on all of $IR(-j)$.

By construction $\overline{\phi}$ is $R$-linear. Moreover, $\overline{\phi}$ is graded. Indeed, let $r\in IR(-j)$ be homogeneous. Then $r=\sum r_i f_i$ with $r_i\in R$ and $f_i\in IR^{(n)}(-j)$ homogeneous with $\deg_{-j}(r_if_i)=\deg_{-j}(r)$. On the other hand, we have that $\overline{\phi}(r)=\sum r_i \phi(f_i)$. Now, $\deg_{-k}(\sum r_i \phi(f_i))=\deg(r_i)+\deg_{-k}(\phi(f_i))$. As $f$ is graded homomorphism, we have that $\deg_{-k}(\overline{\phi}(r))=\deg_{-j}(r)$, which shows that $\overline{\phi}$ is graded. \qedhere

 \end{proof}

An important consequence of Lemma \ref{VUL} is:

\begin{cor}\label{C2}
Fix $n\in \N$, and let $R$ be a positively graded $k$-algebra generated by finitely many elements of degree coprime to $n$. Suppose that $M$ is a *injective $R$-module, and $\Ass_R(M)\subseteq \Proj(R)$. Then $M^{(n)}$ is a *injective $R^{(n)}$-module.
\end{cor}

\begin{proof}
By Theorem \ref{*structure} and assumption, we have that 

\[
M\cong \bigoplus_{\substack{\mathfrak{p} \in \Proj(R) \\ z\in \Z}} \text{*}E_R(R/\mathfrak{p})(-z)^{\oplus\text{*}\eta(i,z,\p,M)},
\]

\

\noindent  By Lemma \ref{*shift} we have that each $\text{*}E_R(R/\mathfrak{p})(-z)^{\text{*}\mu(i,z,\p,M)}$ is *injective. Moreover, by Proposition \ref{*vcwds}

\[
M^{(n)}\cong \bigoplus_{\substack{\mathfrak{p}\in  \Proj(R)\\ z\in \Z}} \text{*}(E_R(R/\mathfrak{p})(-z)^{\oplus\text{*}\eta(i,z,\p,M)})^{(n)}.
\]

\

\noindent By Proposition \ref{VUL} we have that each $\text{*}(E_R(R/\mathfrak{p})(-z))^{(n)}$ is a *injective *$R^{(n)}$-module. As direct sums of *injective modules are *injective, we have that $M^{(n)}$ is *injective, as desired. \qedhere

\end{proof}

The following example shows the assumption in Corollary \ref{C2} that the homogeneous maximal ideal is not an associated prime of $M$ is necessary.

\begin{ex}\label{Jack's Example}
Let $R=k[x,y]$ with the standard grading and consider the injective module $E:=E_R(k)\cong\Hom_k(R,k)$. Lemma 2.3 implies that $E=\text{*}E_R(k)$. By Lemma 2.4, we have that $M:=E(-1)$ is *injective. Also, *$\Ass(M)=\{(x,y)\}$. We show that $M^{(2)}$ is not *injective. 

Let $I=(x^2, xy)(+2)$ be a shifted, graded ideal of $R^{(2)}$. Define the *$R^{(2)}$-linear map 

\[
f:I=(x^2, xy)(+2) \to M^{(2)}
\]

\

\noindent by $f(x^2)=y^*$ and $f(xy)=0^*$, where $y^*\in R^\star$ is defined by $y^*(\sum_{i,j}k_{i,j}x^iy^j)=k_{0,1}$ and $0^*$ is the zero map. It is not hard to see that $f$ is well defined. If, for sake of contradiction, $f$ extended to a map $g:R^{(2)}(+2)\to M^{(2)}$, then 


\[
x^2\cdot g(1)=y^* \text{    and    } xy \cdot g(1)=0^*.
\]

\noindent The first equality implies that $1=(x^2 \cdot g(1))(y)=g(1)(x^2y)$, so $g(1)=(x^2y)^*$. But Then

\[
0=(xy\cdot g(1))(x)=(x^2y)^*(x^2y)=1,
\]

\noindent a contradiction. Thus, $f$ cannot be extended to a $g:R^{(2)}(+2)\to M^{(2)}$. Hence, we see that $M^{(2)}$ is not injective. 

\end{ex}


The following two lemmas are key to the proof of Corollary \ref{EUC1}.

\begin{lem}\label{veronese *essential}
 Fix $n\in \N$, and let $R$ be a positively graded $k$-algebra generated by finitely many elements of degree coprime to $n$. Let $z\in \Z$ and $\p \in \Proj(R)$. Then 

\[
(R/\mathfrak{p})(-z)^{(n)} \subseteq \text{*}E_R((R/\mathfrak{p})(-z))^{(n)}
\]

\noindent is *essential. In particular,

\[
\text{*}E_R((R/\mathfrak{p})(-z))^{(n)}=\text{*}E_{R^{(n)}}((R/\mathfrak{p})(-z)^{(n)}).
\]
\end{lem}

\begin{proof}

Let $m\in \text*E_R((R/\mathfrak{p})(-z))^{(n)}$. Since $(R/\mathfrak{p})(-z) \subseteq \text{*}E_R((R/\mathfrak{p})(-z)$ is *essential there is a homogeneous $y\in R$ such that $ym\in (R/\mathfrak{p})(-z)$ and $ym\neq 0$. 

We set $\deg(y)=j$ and $\deg(m)=dn+z$ for some $d\in \Z$. Since $\p \in \Proj(R)$, there is a homogeneous element, $x\in R\smallsetminus \p$, whose degree is co-prime to $n$. In particular, any power of $x$ is not a zero divisor of $(R/\mathfrak{p})(-z)\subset \text{*}E_R((R/\mathfrak{p})(-z))$, and there exists a positive integer $c_1$ and (possibly negative) integer $c_2$ where we have $c_1\deg(x)+c_2n=-(j+z)$. Set $N=c_1$. Then $x^{N}ym\in (R/\mathfrak{p})(-k)$ is nonzero, and $x^{N}ym\in (R/\mathfrak{p})(-k)^{(n)}$. Thus, $(R/\mathfrak{p})(-k)^{(n)} \subseteq \text{*}E_R((R/\mathfrak{p})(-k))^{(n)}$ is *essential.

Since $(R/\mathfrak{p})(-k)^{(n)} \subseteq \text*E_R((R/\mathfrak{p})(-k))^{(n)}$ is *essential, and $\text*E_R((R/\mathfrak{p})(-k))^{(n)}$ is a *injective *$R^{(n)}$-module by Proposition \ref{VUL}, then we have that $\text*E_R((R/\mathfrak{p})(-k))^{(n)}$ and $ \text{*}E_{R^{(n)}}((R/\mathfrak{p})(-k)^{(n)})$ are isomorphic as $R^{(n)}$-modules  by Theorem \ref{*iso}.\qedhere




\end{proof}

















\begin{rem}\label{*hull structure}
We make an important observation. Let $M$ be a graded $R$-module. Theorem \ref{*structure} implies that 

\[
\text{*}E_R(M)=\bigoplus_{\substack{\mathfrak{p}\in \mathrm{\text{*}Spec}(R) \\ k\in \Z}}\text{*}E_R(R/\mathfrak{p}(-k))^{\oplus\text{*}\mu(0,k,\p,M)}.
\]

\

\noindent Fix a $\text{*}E_R(R/\mathfrak{p}(-k))$ appearing in the decomposition above. There is then an isomorphic copy of $R/\mathfrak{p}(-k)$ in *$E_R(M)$ such that $R/\mathfrak{p}(-k)\subseteq \text{*}E_R(R/\mathfrak{p}(-k))$ is *essential as submodules of *$E_R(M)$. Let $x\in\text{*}E_R(R/\mathfrak{p}(-k))\subseteq \text{*}E_R(M)$. Since we have $R/\mathfrak{p}(-k)\subseteq \text{*}E_R(R/\mathfrak{p}(-k))$ is *essential inside of *$E_R(M)$, there is a homogeneous $r\in R$ such that $rx\neq 0$ and $rx\in R/\mathfrak{p}(-k)\subseteq E_R(M)$. Since $M\subseteq \text{*}E_R(M)$ is *essential, there is a homogeneous $r'\in R$ such that $m:=r'rx$ is a nonzero element of $M$. Thus, $\ann_R(m)=\p$.

\

To summarize, for every $x$ in the copy of $\text{*}E_R(R/\mathfrak{p}(-k))\subseteq \text{*}E_R(M)$, there is a homogeneous $s\in R$ such that $sx=m$ is a nonzero element of $M$ and $\ann(m)=\mathfrak{p}$.
\end{rem}

\begin{lem}\label{KL}
Fix $n\in \N$, and let $R$ be a positively graded $k$-algebra generated by finitely many elements of degree coprime to $n$. Assume that $M$ is a graded $R$-module such that $\Ass_R(M)\subseteq \Proj(R)$. Then

\[
M^{(n)} \subseteq (\text{*}E_R(M))^{(n)}
\]

\

\noindent is *essential.
\end{lem}

\begin{proof}
By Theorem \ref{*structure} we may write

\[
\text{*}E_R(M)\cong \bigoplus_{\substack{\mathfrak{p}\in \Proj(R) \\ k\in \Z}} \text{*}E_R(R/\mathfrak{p}(-k))^{\oplus\text{*}\mu(i,k,\p,M)},
\] 

\

\noindent Proposition \ref{*vcwds} yields

\[
(\text{*}E_R(M))^{(n)}\cong \bigoplus_{\substack{\mathfrak{p}\in \Proj(R) \\ k\in \Z}} \text{*}(E_R(R/\mathfrak{p})(-k)^{\oplus\text{*}\mu(0,k,\p,M)})^{(n)}.
\]

\

\noindent By Proposition \ref{veronese *essential}, we then have

\[
(\text{*}E_R(M))^{(n)}\cong \bigoplus_{\substack{\mathfrak{p}\in \Proj(R) \\ k\in \Z}} \text{*}E_{R^{(n)}}((R/\mathfrak{p}(-k))^{(n)})^{\oplus\text{*}\mu(0,k,\p,M)}.
\]

\

\noindent Let $x:=(x_1,x_2, \hdots, x_m,0,\hdots)\in \text{*}E_R(M)^{(n)}$ be nonzero. We show by induction on $m$, that $R^{(n)}x\cap M^{(n)}\neq 0$. Let $m=1$, and let $x_1\in \text{*}E_{R^{(n)}}((R/\mathfrak{p})(-k))^{(n)}$. By Remark \ref{*hull structure} , there is a homogeneous $r\in R$ such that $rx_1\in M$ and $\ann(rx_1)=\mathfrak{p}$. Let $\deg(r)=j$. Since $\p\in \Proj(R)$, there is a generator $s\in R\setminus \p$ of degree coprime to $n$. In particular, there is a positive integer $c_1$ and (possibly negative) integer $c_2$ such that $c_1\deg(s)+c_2n=-j$. Set $M=c_1$. Then $s^Mrx_1\neq 0$ and is an element of $M^{(n)}$. This completes the base case.

 Suppose for $m\leq n$ and $x:=(x_1,x_2, \hdots, x_m,0,\hdots)\in (\text{*}E_R(M))^{(n)}$, there is a homogeneous $r\in R^{(n)}$ such that $rx\in M^{(n)}$ is nonzero. Let $y:=(y_1,y_2,\hdots, y_{n+1},0,\hdots)$. Set $y':=(y_1,y_2,\hdots, y_n,0,0\hdots)$ and $y''=(0,\hdots, 0, y_{n+1}, 0,\hdots )$. If $y'$ is zero, then we are finished by the base case. Suppose $y'\neq 0$; then by the induction hypothesis there is a homogeneous $r\in R^{(n)}$ such that $ry'$ is nonzero in $M^{(n)}$. If $ry_{n+1}=0$, then we have $ry=ry'$ is nonzero in $M^{(n)}$. On the other hand, if $ry_{n+1}\neq 0$, then we may write $ry=ry'+ry''$ with $ry'\in M^{(n)}$ nonzero. By the base case, there is a homogeneous $r'\in R^{(n)}$ such that $r'ry''$ is nonzero in $M^{(n)}$. Thus, $r'ry=r'ry'+r'y''$ is nonzero in $M^{(n)}$. By induction, we conclude  

 \[
M^{(n)} \subseteq (\text{*}E_R(M))^{(n)}
\]

\noindent is *essential.
\end{proof}

\begin{cor}\label{EUC1}
Fix $n\in \N$, and let $R$ be a positively graded $k$-algebra generated by finitely many elements of degree coprime to $n$. Assume that $M$ is a graded $R$-module such that $\Ass_R(M)\subseteq \Proj(R)$. Then

\[
\text{*}E_{R^{(n)}}(M^{(n)}) =  (\text*E_R(M))^{(n)}.
\]

\end{cor}

\begin{proof}
This is an immediate consequence of Lemma \ref{KL} and Corollary \ref{C2}.
\end{proof}

By definition the degrees of the generators of a standard graded, finitely generated $k$-algebra $R$ are all $1$. Hence, their degrees are coprime to $n\in\N$. We single out this case.

\begin{cor}\label{EUC2}
Fix $n\in \N$, and let $R$ be a standard graded, finitely generated algebra over a field. If $M$ is a graded $R$-module such that $\Ass_R(M)\subseteq \Proj(R)$, then

\[
\text{*}E_{R^{(n)}}(M^{(n)}) =  (\text*E_R(M))^{(n)}.
\]

\end{cor}

We will find use for the following Proposition in Section \ref{sec5}.

\begin{prop}\label{vup}
Fix $n\in \N$, and let $R$ be a positively graded $k$-algebra generated by finitely many elements of degree coprime to $n$. If $\p\in \Proj(R)$, then 

\[
\text{*}\mu(0,\q\cap R^{(n)},(R/\p(-k))^{(n)})=\begin{cases}
0, & \q\neq \p\\
1, & \q=\p.
\end{cases}
\]
\end{prop}

\begin{proof}
 Since $\text{*}\mu(0, \q\cap R^{(n)},(R/\p(-k))^{(n)})=\mu(0, \q\cap R^{(n)},(R/\p(-k))^{(n)})$, to show the desired equality it suffices to show there is an essential (not necessarily graded) extension $R^{(n)}/(\p\cap R^{(n)})\to (R/\p(-k))^{(n)}$. We set 
 \[
 j=\min\{i~|~[(R/\p(-k))^{(n)}]_i\neq 0\}.
\] 
The set of monomials in $[(R/\p(-k))^{(n)}]_j=\{y_1,\hdots,y_m\}$ generate $R/\p(-k)^{(n)}$ as a (ungraded) $R^{(n)}$-module  Consider the injection $R^{(n)}/\q \cap R^{(n)}\to (R/\p(-k))^{(n)}$ defined by sending $1$ to $y_1$. To show that this is essential, it suffices to show that 
 
 \[
 \left[R^{(n)}/(\p\cap R^{(n)})\right]y_1 \cap \left[R^{(n)}/(\p\cap R^{(n)})\right]h\neq 0
 \]

 \
 
 \noindent for all nonzero homogeneous $h\in (R/\p(-k))^{(n)}$. Consider, $y_1\cdot h^n=h\cdot(y_1\cdot h^{n-1})$. Since $y_1\cdot h^{n-1}$ is an element of $R^{(n)}/(\p\cap R^{(n)})$, we attain the desired statement.
\end{proof}
\end{subsection}
\end{section}


\begin{section}{Graded Matlis Duality}\label{sec4}
In this section we investigate how the graded Matlis Functor can be used to turn minimal *free resolutions into minimal *injective resolutions under certain circumstances. This will allow us to understand the graded bass numbers of a graded module over the homogeneous maximal ideal of a *local ring in terms of its graded Betti numbers. To begin, we recall what the graded Matlis functor is; a nice treatment of this can be found in \cite{BH}. 

\begin{definition}
Let $R$ be a *local ring with homogeneous maximal ideal $\m$. Set $R_0=[R]_0$ and $\m_0=\mathfrak{m}\cap R_0$. The graded Matlis functor from graded $R$-modules to graded $R$-modules is defined to be

\[
(\--)^{\vee}:=\underline{\Hom}_R(\--,~ E_{R_0}(R_0/\m_0)).
\]

\noindent Note that $(\--)^{\vee}$ is an exact additive functor.
\end{definition}

\begin{definition}
Let $R$ be a *local ring with homogeneous maximal ideal $\m$. We define the following functor from the category of graded $R$-modules to the category of graded $R$-modules by

\[
(\--)^\checkgr:=\SHom_R(\--, \text{*}E_R(R/\m)).
\]

\noindent Note that $(\--)^\checkgr$ is an exact additive functor.
\end{definition}

\begin{definition}
We say that a *local ring is a *complete *local ring if $(R_0,\m_0)$ is a complete local ring. 
\end{definition}

In general $(M)^\vee$ and $(M)^\checkgr$ are not isomorphic as graded $R$-modules; however, if $(R,\m)$ is *complete, then they are:

\begin{thm}[{\cite[Proposition 3.3.16, Theorem 3.6.17]{BH}}]\label{Matlis Duality}
Let $(R,\m)$ be a *complete *local ring. Suppose $M$ is a *Noetherian $R$-module and $N$ is a *Artinian $R$-module. Then

\begin{enumerate}
\item $(M)^\vee \cong (M)^\checkgr$ as graded $R$-modules for all graded $R$-modules $M$;
\item $M^\vee$ is *Artinian;
\item $N^\vee$ is *Noetherian;
\item $M^{\vee\vee}\cong M$ and $N^{\vee\vee}\cong N$.
\end{enumerate}
\end{thm}

For what is to follow, we utilize a graded version of Ext-Tor duality

\begin{lem}\label{graded duality}
Let $R$ be a *Noetherian graded ring and $E$ any *injective $R$-module. Set 

\[
F(\--):=\SHom_R(\--,E). 
\]

\noindent Then 

\begin{enumerate}
\item For all graded $R$-modules $M$ and $N$, we have $F(\uTor_R^i(M,N))\cong \uExt_i^R(M,F(N))$.
\item For all graded $R$-modules $M$ and $N$ with $M$ *Noetherian, we have an isomorphism of graded $R$-modules $F(\uExt_i^R(M,N))\cong \uTor_R^i(M,F(N))$.
\end{enumerate}
\end{lem}

\begin{proof}
This follows as in the ungraded case.
\end{proof}

We now prove the main theorem of this section.

\begin{thm}\label{gt}
Let $(R,\m)$ be *local ring with $R/\m=k$ a field. Suppose $M$ is a *Noetherian $R$-module with minimal *free resolution $F_\bullet$. Then $(F_\bullet)^\checkgr$ is a minimal *injective resolution of $M^\checkgr$. In particular, $\beta(i,M,-z)=\text{*}\mu(i,\m,z,M^\checkgr)$ for all $i$ and all $z\in \Z$, where $\beta(i,M,-z)$ is the graded $(i,-z)$-th Betti number for $M$.
\end{thm}

\begin{proof}
Considered the augmented complex

\[
 F_\bullet\to M:= \cdots \to \bigoplus_{z\in \Z} R(z)^{\oplus \beta_{1,M,-z}}\to \bigoplus_{z\in \Z} R(z)^{\oplus\beta_{0,M,-z}}\to M \to 0,
\]

\

\noindent which is exact. As $(\--)^\checkgr$ is an exact functor, the complex

\[
(M\to F_{\bullet})^\checkgr= M^\checkgr\to \bigoplus_{z\in \Z} E_R(k)(-z)^{\oplus\beta_{0,M,-z}} \to \bigoplus_{z\in \Z} R(-z)^{\oplus\beta_{1,M,-z}} \to \cdots
\]

\

\noindent is exact. Therefore, $(F_\bullet)^\checkgr$ is a *injective resolution of $M^\checkgr$. Set $E^\bullet:=(F_\bullet)^\checkgr$. We observe, $E^i$ is $\m$- torsion for all $i$. Thus, the complex $\underline{\Hom}_{R_{\p^*}}\left(~R_{(\p)}/\p\ R_{(\p)},~\left(E^\bullet\right)_{\p^*}\right)$ is 0 for all $\p\in \text{*}\Spec(R)\setminus \{\m\}$. Hence, the differential of $\underline{\Hom}_{R_{\p^*}}\left(~R_{(\p)}/\p\ R_{(\p)},~\left(E^\bullet\right)_{\p^*}\right)$ is 0 for all $\p\in \text{*}\Spec(R)\setminus \{\m\}$.

Next, since $F_\bullet$ is a minimal *free resolution of $M$, $R/\m \otimes F_\bullet$ has zero differential. Using Hom-Tensor adjunction for chain complexes, we see that $\Hom_R(R/\m, E^\bullet)$ has zero differential. Since all *Bass numbers of $M^\checkgr$ are finite, Theorem \ref{classification of minimal} implies,  that $(F_\bullet)^\vee$ is a minimal *injective resolution of $M^\vee$, as desired.
\end{proof}

The following example demonstrates how Theorem \ref{gt} can be used.

\begin{ex}\label{ex1}

Let $k$ be a field and set $R=k[x,y]$ with the standard grading. Using Macaulay2 \cite{M2},  one can show that  $R^{(3)}$ is a Golod ring. This implies that its Poincaré series is

\[
P_k(t)=\frac{(1+t)^4}{1-(3t^2+2t^3)}=1+4t+\sum_{k=2}^\infty 9\cdot 2^{k-2}t^k.
\]

\

\noindent In particular, the Betti numbers of $k$ over $R^{(3)}$ are $\beta_0=1$, $\beta_1=4$, and $\beta_n=9\cdot 2^{n-2}$, for $n\geq 2$. Using Theorem \ref{gt},  we attain

\[
\mu( \m \cap R^{(3)}, i, R^{(3)})=\begin{cases}
1 & i=0 \\
4 & i=1\\
9\cdot 2^{2-i} & i\geq 2
\end{cases},
\]
\noindent where $\m$ is the homogeneous maximal ideal of $R$.
\end{ex}

If $(R,\m)$ is *local with $R/\m$ a field then $(R,\m)$ is *complete. So, Theorem \ref{Matlis Duality} and Theorem \ref{gt} yield:

\begin{cor}\label{gc}
Let $(R,\m)$ be *local ring with $R/\m=k$ a field. If $M$ is a *Artinian, then all *Bass numbers of $M$ are finite.
\end{cor}

The next couple of Lemmas will be used in Section \ref{sec5}; however, we include them here since Lemma \ref{fmt} is a direct consequence of Corollary \ref{gc}.

\begin{lem}\label{Artinian and n}
Let $R$ be any graded ring. If $M$ is a *Artinian graded $R$-module, then $M^{(n)}$ is a *Artinian *$R^{(n)}$-module. 
\end{lem}

\begin{proof}
Let $N_0\supset N_1 \supset \cdots$ be a chain of graded $R^{(n)}$-submodules of $M^{(n)}$. For each natural number $i$, we defined $RN_i$ to be the graded $R$-submodule of $M$ generated by the elements of $N_i$. We expand to a chain of graded submodules in $M$:
\[
RN_0 \supseteq RN_1 \supseteq \cdots.
\]

Since $M$ is Artinian, there is a $t\in \N$ such that if $i>j\geq t$, then $RN_i=RN_j$. We show that if $i>j\geq t$, then $N_i=N_j$, as well. Suppose that $m\in N_j$ is homogeneous. Since $N_j\subseteq M^{(n)}$, we note that $\deg(m)$ is divisible by $n$. Since we have $RN_i=RN_j$, there are homogeneous $s_i\in R$ and $x_i\in RN_i$ such that $m=\sum s_ix_i$. Since $\deg(m)$ and $\deg(x_i)$ are divisible by $n$, we have that $\deg(s_i)$ is divisible by $n$. Therefore, it follows $s_i\in R^{(n)}$. Thus, $N_i=N_j$, as desired.
\end{proof}

\begin{lem}\label{fmt} Suppose that $(R,\m,k)$ is a *local ring. If $M$ is a *Artinian $R$-module, then $M^{(n)}$ has finite *Bass numbers over $R^{(n)}$.
\end{lem}

\begin{proof}
This is an immediate consequence of Corollary \ref{gc} and Lemma \ref{Artinian and n}.
\end{proof}

\end{section}

\begin{section}{Finiteness of *Bass Numbers of Veronese Submodules} \label{sec5}

For this section, $(R,\m, k)$ will always denote a *local ring with unique homogeneous maximal ideal $\m$. We set the convention that for a graded $(R,\m,k)$-module $M$ with minimal *injective resolution $(\text{*}E^\bullet_M, \del_M)$, where the $i$-th differential is the map $\del_M^i: \text{*}E_M^{i-1}\to \text{*}E_M^i$. Moreover, we set $E^{-1}_M=M$. In this section, we prove Theorem \ref{main theorem 2}, of which Theorem \ref{main theorem 1} is a corollary. For the proof of Theorem \ref{main theorem 2}, we will find it helpful to develop some convenient notation.

\begin{definition}\label{mtorchar}
Let $(R,\m, k)$ be a *local ring. The \textit{$\m$-torsion indicator function}, which we denote by $\bigchi_\m: \text{*$R$-Mod} \to \N$, is defined by
\[
\bigchi_\m\left(M\right)=\min\{i~|~ \Gamma_\m(E_M^i)\neq 0\}=\min\{i~|~\mu(\m,i,M)\neq0\}
\]

\noindent for each graded $R$-module $M$, with *$E^\bullet_M$ the minimal *injective resolution of $M$. Moreover, we set

\[
A(M)= \coker\left(\del_M^{\chi_\m(M)-1}\right).
\]
\end{definition}

Before proving Theorem \ref{main theorem 2}, we will prove the following lemma.



\begin{lem}\label{clemma}
Let $M$ be a graded $(R,\m,k)$-module.

\begin{enumerate}
\item[(a)] $\text{*}\mu\left(\p,j,\frac{M}{\Gamma_\m(M)}\right)=\text{*}\mu(\p, j, M)$ for all $\p\in \Proj(R)$ and $j\in \N$. 
\item[(b)] $\text{*}\mu\left(\p\cap R^{(n)},j, \left(\frac{M}{\Gamma_\m(M)}\right)^{(n)}\right)=\text{*}\mu(\p\cap R^{(n)}, j, M^{(n)})$ for all $j\in \N$. 
\end{enumerate}
\end{lem}

\begin{proof}
\begin{enumerate}
\item[(a):] Consider the short exact sequence 

\begin{equation*}
0 \to \Gamma_\m(M)\to M\to \frac{M}{\Gamma_\m(M)}\to 0.
\end{equation*}

\noindent Applying $(\--)_{(\p)}$, with $\p\in \Proj(R)$, to the short exact sequence above yields
\[
M_{(\p)}\cong \left(\frac{M}{\Gamma_\m(M)}\right)_{(\p)}.
\]

\noindent Therefore, by Lemma \ref{*injective localizes} and Lemma \ref{*minimal localizes}, $\text{*}\mu\left(\p,j,\frac{M}{\Gamma_\m(M)}\right)=\text{*}\mu(\p, j, M)$ whenever $\p\in \Proj(R)$ and $j\in \N$.

\item[(b):] The proof is analogous to the argument of part (a) after applying $(\--)_{(\p\cap R^{(n)})}$ to the short exact sequence 

\[
0 \to \Gamma_\m(M)^{(n)}\to M^{(n)} \to \left(\frac{M}{\Gamma_\m(M)}\right)^{(n)}\to 0.\qedhere
\]

\end{enumerate}
\end{proof}

\begin{thm}\label{main theorem 2}
Fix $n\in \N$, and let $R$ be a positively graded $k$-algebra generated by finitely many elements of degree coprime to $n$. Let $M$ be a graded $R$-module with finite Bass numbers over $R$. Then $M^{(n)}$ has finite Bass numbers over $R^{(n)}$. Moreover, if $\p\in \Proj(R)$, then for all $i$, we have

\begin{equation}\label{eq25}
\text{*}\mu(\p,i,M)=\text{*}\mu(\p\cap R^{(n)}, i,M^{(n)}).
\end{equation}
\end{thm}

\begin{proof}
 We proceed by induction on $i$. Let $i=0$. We consider two cases: $\bigchi_\m(M)\geq 1$ and $\bigchi_\m(M)=0$. If $\bigchi_\m(M)\geq 1$, then $\Ass_R(M)\subseteq \Proj(R)$; thus, by Corollary \ref{EUC1} we have $(\text{*}E^0_M)^{(n)}=\text{*}E_{M^{(n)}}^0$. Applying Proposition \ref{vup} yields
\[
\text{*}\mu(\p, 0, M)=\text{*}\mu(\p\cap R^{(n)}, 0, M^{(n)}),
\]

\noindent for all $\p\in \Proj(R)$. Moreover, by Corollary \ref{EUC2}, $\text{*}\mu(\m\cap R^{(n)},0,M^{(n)})=0$. This completes the first case.

Next, Suppose that $\bigchi_\m(M)=0$. Consider the short exact sequence 

\begin{equation}\label{ses1}
0 \to \Gamma_\m(M)\to M \to \frac{M}{\Gamma_\m(M)}\to 0.
\end{equation}

\noindent  In this case $A(M)=M$. So, Lemma \ref{clemma} implies $\text{*}\mu(\p,i,M)=\text{*}\mu\left(\p, i , \frac{M}{\Gamma_\m(M)}\right)$ for all $i$ and for all $\p\in \Proj(R)$. Since $\bigchi_\m\left(\frac{M}{\Gamma_\m(M)}\right)\geq 1$ (it has no $\m$-torsion elements), the first case implies

\begin{equation}\label{eq22}
\text{*}\mu(\p,0,M)=\text{*}\mu\left(\p, 0, \frac{M}{\Gamma_\m(M)}\right)=\text{*}\mu\left(\p \cap R^{(n)}, 0, \left(\frac{M}{\Gamma_\m(M)}\right)^{(n)}\right).
\end{equation}

\noindent On the other hand, consider the short exact sequence

\begin{equation}\label{ses2}
0\to \Gamma_\m(M)^{(n)} \to M^{(n)} \to \left(\frac{M^{(n)}}{\Gamma_\m(M)^{(n)}}\right) \to 0.
\end{equation}

\noindent Lemma \ref{clemma} implies $\text{*}\mu(\p\cap R^{(n)}, i, M^{(n)})=\text{*}\mu\left(\p \cap R^{(n)}, i, \left(\frac{M}{\Gamma_\m(M)}\right)^{(n)}\right)$ for all $i$. Therefore, Equation \ref{eq22} implies that
\[
\text{*}\mu(\p,0,M)=\text{*}\mu(\p\cap R^{(n)}, 0, M^{(n)}).
\]

\noindent This proves (\ref{eq25}) for the second case. Additionally, since $\Gamma_\m(M)$ is *Artinian over the ring $R$, then by Lemma \ref{fmt} we have that $\Gamma_\m(M)^{(n)}$ has finite Bass numbers over $R^{(n)}$. Applying $\Hom_{R^{(n)}}(\frac{R^{(n)}}{\m\cap R^{(n)}}, \--)$  to the short exact sequence in $(\ref{ses2})$, yields the left exact sequence

\[
\begin{aligned}
0\to \Hom_{R^{(n)}}\left(\frac{R^{(n)}}{\m\cap R^{(n)}}, \Gamma_\m(M)^{(n)}\right) &\to \Hom_{R^{(n)}}\left(\frac{R^{(n)}}{\m\cap R^{(n)}}, M^{(n)}\right)\\ &\to \Hom_{R^{(n)}}\left(\frac{R^{(n)}}{\m\cap R^{(n)}}, \frac{M^{(n)}}{\Gamma_\m(M)^{(n)}}\right).
\end{aligned}
\]

\noindent By Corollary \ref{*hull localization 1}, $\Hom_{R^{(n)}}\left(\frac{R^{(n)}}{\m\cap R^{(n)}},\frac{ M^{(n)}}{\Gamma_\m(M)^{(n)}}\right)=0$. Hence, 

\[
\Hom_{R^{(n)}}\left(\frac{R^{(n)}}{\m\cap R^{(n)}}, \Gamma_\m(M)^{(n)}\right) \cong \Hom_{R^{(n)}}\left(\frac{R^{(n)}}{\m\cap R^{(n)}}, M^{(n)}\right).
\]

\noindent In particular, $\text{*}\mu(\m \cap R^{(n)}, 0, \Gamma_\m(M)^{(n)})=\text{*}\mu(\m\cap R^{(n)}, 0, M^{(n)})$, implying that the number $\text{*}\mu(\m\cap R^{(n)}, 0, M^{(n)})$ is finite. This completes the base case.

Inductively, assume for all graded $R$-modules, say N, with finite Bass numbers, that for all $\p\in \Proj(R)$, and $0\leq i <m$, we have
\[
\text{*}\mu(\p, i, N) =\text{*}\mu(\p\cap R^{(n)}, i, N^{(n)}).
\]

\noindent Furthermore assume that $\text{*}\mu(\m\cap R^{(n)}, i, N^{(n)})$ is finite whenever $0\leq i<m$. Let $M$ be any graded $R$-module with finite Bass numbers and $\p \in \Proj(R)$. We split into cases: $\bigchi_\m(M)\geq 1$ and $\bigchi_\m(M)=0$. First, assume that $\bigchi_\m(M)\geq 1$. Suppose that $m < \bigchi_\m(M)$, then $E_M^m$ does not have $\m$ as an associated prime. Therefore, by Corollary \ref{EUC1} and Proposition \ref{vup}, it follows that $\text{*}\mu(\m\cap R^{(n)}, m, M^{(n)})=0$ and that $\text{*}\mu(\p, m, M)=\text{*}\mu(\p \cap R^{(n)}, m, M^{(n)})$.

Now, suppose that $\bigchi_\m(M)\leq m$. Then as 

\[
0 \to (E_M^0)^{(n)} \xra{(\del_M^1)^{(n)}} (E_M^1)^{(n)} \xra{(\del_M^2)^{(n)}} \cdots \xra{(\del_M^{\chi_\m(M)-1})^{(n)}} (E_{M}^{\chi_\m(M)-1})^{(n)}
\]

\

\noindent is the start of a minimal *injective resolution of $M^{(n)}$, we see that 

\[
\text{*}\mu(\p\cap R^{(n)}, i, M^{(n)})=\text{*}\mu\left(\p \cap R^{(n)}, i-\bigchi_\m(M), A(M)^{(n)}\right)
\]

\noindent for all $i\geq \bigchi_\m(M)$ and $\p\in \text{*}\Spec(R)$. Moreover, 

\[
\text{*}\mu(\p, i, M)=\text{*}\mu(\p, i-\bigchi_\m(M), A(M))
\]

\noindent for all $i\geq \bigchi_\m(M)$. Note that $A(M)$ has finite Bass numbers over $R$. Since we have that $\bigchi_\m(M) \leq m$, upon applying the induction hypothesis to $A(M)$, for any prime ideal $\p\in \Proj(R)$, we obtain

\[
\begin{aligned}
\text{*}\mu(\p, m, M)&=\text{*}\mu(\p, m-\bigchi_\m(M), A(M))\\
&= \text{*}\mu\left(\p \cap R^{(n)}, m-\bigchi_\m(M), A(M)^{(n)}\right)\\
&=\text{*}\mu(\p\cap R^{(n)}, m, M^{(n)}).
\end{aligned}
\]

\

\noindent On the other hand, applying the induction hypothesis to $A(M)$, with $\p=\m$, implies

\[
\text{*}\mu(\m, i, M)=\text{*}\mu(\m, i-\bigchi_\m(M), A(M))
\]

\noindent is finite. This completes the first case.

Now, suppose that $\bigchi_\m(M)=0$. In this case, we have that $A(M)=M$. Since the quotient $\frac{M}{\Gamma_\m(M)}$ does not have $\m$ as an associated prime, then $\bigchi_\m\left(\frac{M}{\Gamma_\m(M)}\right)\geq 1$. We note that all Bass numbers over $R$ of $\frac{M}{\Gamma_\m(M)}$ are finite.To see this, we first note that by Lemma \ref{clemma}, if $\p\in \Proj(R)$, then all $\mu\left(\p,i,\frac{M}{\Gamma_\m(M)}\right)$ are finite. Since $\Gamma_\m(M)$ and $M$ have finite Bass numbers over $R$ and we have a long exact sequence

\[
\cdots \to \Ext_R^{i}(k, M)\to \Ext_R^i\left(k,\frac{M}{\Gamma_\m(M)} \right) \to \Ext_R^{i+1}(k,\Gamma_\m(M))\to \cdots
\]

\noindent it follows that $\mu\left(\m,i,\frac{M}{\Gamma_\m(M)}\right)$ are finite. Therefore, we may apply the first case of the induction step to $\frac{M}{\Gamma_\m(M)}$. If $\p\in \Proj(R)$, then by Lemma \ref{clemma}

\[
\begin{aligned}
\text{*}\mu(\p,m,M)&=\text{*}\mu\left(\p, m, \frac{M}{\Gamma_\m(M)}\right)\\
&=\text{*}\mu\left(\p \cap R^{(n)}, m, \left(\frac{M}{\Gamma_\m(M)}\right)^{(n)}\right)\\
&=\text{*}\mu(\p\cap R^{(n)}, m, M^{(n)}).
\end{aligned}
\]

 On the other hand, if $\p=\m$, then the first case of the induction step applied to the module $A(M)$ implies 

\[
\text{*}\mu \left(\m\cap R^{(n)}, m, \left(\frac{M}{\Gamma_\m(M)}\right)^{(n)}\right)
\]

\noindent is finite. Moreover, by Lemma \ref{fmt} $\text{*}\mu(\m\cap R^{(n)}, i, \Gamma_\m(M)^{(n)})$ is finite for all $i$.

The long exact sequence of $\Ext_{R^{(n)}}^\bullet\left(\frac{R^{(n)}}{\m\cap R^{(n)}}, \--\right)$ induced by the short exact sequence $0\to \Gamma_\m(M)^{(n)} \to M^{(n)} \to \left(\frac{M^{(n)}}{\Gamma_\m(M)^{(n)}}\right) \to 0$ yields the exact sequence

\[
\begin{aligned}
\Ext_{R^{(n)}}^m\left(\frac{R^{(n)}}{\m\cap R^{(n)}}, \Gamma_\m(M)^{(n)}\right)&\to \Ext_{R^{(n)}}^m\left(\frac{R^{(n)}}{\m\cap R^{(n)}}, M^{(n)}\right)\\
&\to \Ext_{R^{(n)}}^m\left(\frac{R^{(n)}}{\m\cap R^{(n)}}, \left(\frac{M}{\Gamma_\m(M)}\right)^{(n)} \right).
\end{aligned}
\]

\noindent The spaces $\Ext_{R^{(n)}}^m\left(\frac{R^{(n)}}{\m\cap R^{(n)}}, \Gamma_\m(M)^{(n)}\right)$ and $\Ext_{R^{(n)}}^m\left(\frac{R^{(n)}}{\m\cap R^{(n)}}, \left(\frac{M}{\Gamma_\m(M)}\right)^{(n)} \right)$ are finite-dimensional $k$-vector spaces, so we conclude that the $k$-vector space dimension of $\Ext_{R^{(n)}}^m\left(\frac{R^{(n)}}{\m\cap R^{(n)}}, M^{(n)}\right)$ is also finite. In particular $\text{*}\mu(\m \cap R^{(n)}, m, M^{(n)})$ is finite; thereby, completing the induction step and hence the proof.\qedhere

\end{proof}







We now give the proof of Theorem \ref{main theorem 1}.

\begin{thm}\label{m}
Fix $n\in \N$, and let $R$ be a positively graded $k$-algebra generated by finitely many elements of degree coprime to $n$. Suppose that $M$ is a graded $R$-module with finite Bass numbers over $R$. Then $M^{(n)}$ has finite Bass numbers over $R^{(n)}$. Moreover, if $\p\in \Spec(R)$ and $\p$ is not the homogeneous maximal ideal, then

\[
\mu(i,\p \cap R^{(n)}, M^{(n)})=\begin{cases}
\mu(i,\p, M) & \p \in \Proj(R), \\
0 & \text{$\p$ is not homogeneous and $i=0$,} \\
\mu(i-1,\p^*, M) & \text{or $\p$ is not homogeneous and $i\geq 1$},
\end{cases}
\]

\

\noindent where $\p^*=(\{r\in \p~|~ \text{$r$ is homogeneous}\})$ is the largest homogeneous prime ideal contained in $\p$.
\end{thm}

\begin{proof}
This follows by Theorem \ref{main theorem 2} and Propositions \ref{graded bass numbers} and \ref{grade=}.
\end{proof}

\end{section}


\begin{section}{An Application to Local Cohomology} \label{sec6}
In this section, we illustrate an application of Theorem \ref{main theorem 1} to local cohomology modules. We begin by recalling the definition of local cohomology modules over a Noetherian ring.

\begin{definition}
Fix a Noetherian ring $R$, ideal $I=(f_1,\hdots, f_n)$ of $R$, and $R$-module $M$. We define the \textit{$I$-torsion functor} to be the functor $\Gamma_I: R-\text{Mod}\to R-\text{Mod}$, where for an $R$-module $M$

\[
\Gamma_I(M)=\{m\in M~|~ \text{$I^t\cdot m=0$ for some $t\in \N$}\},
\]

\

\noindent and for a map of $R$-modules $f:M\to N$, we define $\Gamma_I(f)$ to be the restriction of $f$

\[
\Gamma_I(f)=f:\Gamma_I(M) \to \Gamma_I(N).
\]

\

\noindent Let $E_M^\bullet$ be an injective resolution of $M$ over $R$. Define the \textit{$i$-th local cohomology of $M$ supported on $I$}, denoted $H^i_I(M)$, to be the $i$-th cohomology of $\Gamma_I(E_M^\bullet)$. 
\end{definition}

\

The following proposition and its proof can be found in \cite[Corollary 5.6]{brodmann_local_1984}.

\begin{prop}\label{*cohomology}
Let $R$ be any $\Z$-graded Noetherian ring and let $I\subseteq R$ be a graded ideal. Then for all $i$ and $n\in \N$, there is a graded $R^{(n)}$ isomorphism

\[
H_I^i(R)^{(n)}\cong H_{I\cap R^{(n)}}^i(R^{(n)}).
\]

\end{prop}

Corollaries \ref{graded local cohomology thm 2} and \ref{graded local cohomology cor 2} are a consequence of Theorem \ref{main theorem 1} and Proposition \ref{*cohomology} from the introduction.

\begin{cor}\label{cor1}
Fix $n\in \N$, and let $R$ be a positively graded $k$-algebra generated by finitely many elements of degree coprime to $n$. Let $I$ be an ideal of $R$ and $M$ be a graded $R$-module. If $H_I^i(M)$ has finite Bass numbers over $R$, then $H_{I\cap R^{(n)}}^i(M^{(n)})$ has finite Bass numbers over $R^{(n)}$. Moreover,

\[
\mu(i,\p\cap R^{(n)}, H_{I\cap R^{(n)}}^j(M^{(n)}))=\begin{cases}
\mu(i,\p, H_{I}^j(M)) & \p\in \Proj(R), \\
0 & \text{$\p$ is not homogeneous and $i=0$,} \\
\mu(i-1,\p^*, H_{I}^j(M)) & \text{or $\p$ is not homogeneous and $i\geq 1$}
\end{cases},
\]

\

\noindent where $\p^*=(\{r\in \p~|~ \text{$r$ is homogeneous}\})$.

\end{cor}

\begin{cor}\label{cor2}
With the same setup as in Corollary \ref{cor1}, we have

\[
\Ass_{R^{(n)}}\left(H^i_{I\cap R^{(n)}}\left(M^{(n)}\right)\right)=\left\{\p\cap R^{(n)}~|~ \p \in \Ass_R(H_I^i(M))\right\}.
\]
\end{cor}

We recall that Corollary \ref{cor2} can be deduced from   \cite[Corollary 1.10]{YassemiSiamak1998Wapu}; our method is different as it appeals to \ref{cor1}.






\end{section}


\bibliographystyle{amsplain}
\bibliography{main}

\end{document}